\documentclass[a4paper, 12pt]{article}
\usepackage[english]{babel}
\usepackage{xcolor}
\usepackage{xypic}
\usepackage[utf8]{inputenc}
\usepackage[T1]{fontenc}
\usepackage{lmodern}
\usepackage{amsmath}
\usepackage{amsthm}
\usepackage{mathtools}
\usepackage{amssymb}
\usepackage{tikz-cd}
\usepackage{slashed}
\usepackage{parskip}
\usepackage{amsfonts}
\usepackage{thmtools,xcolor}
\usepackage{hyperref}
\usepackage{cleveref}
\usepackage{mathrsfs}
\usepackage{enumitem}
\usepackage{subfiles}
\usepackage{imakeidx}
\usepackage{microtype}
\usepackage[toc,page]{appendix}
\usepackage{crossreftools}
\usepackage{graphicx}
\usepackage{numprint}
\usepackage{setspace}
\usepackage{accents}
\usepackage{emptypage}
\allowdisplaybreaks[1]

\usepackage[a4paper]{geometry}
\theoremstyle{plain}
\newtheorem{theorem}{Theorem}[section]
\newtheorem{cor}[theorem]{Corollary}
\newtheorem{lem}[theorem]{Lemma}
\newtheorem{prop}[theorem]{Proposition}

\theoremstyle{definition}
\newtheorem{examp}[theorem]{Example}
\newtheorem{examps}[theorem]{Examples}
\newtheorem{dfn}[theorem]{Definition}
\newtheorem{rem}[theorem]{Remark}
\newtheorem{rems}[theorem]{Remarks}

\theoremstyle{remark}

\newcommand{\RNum}[1]{\uppercase\expandafter{\romannumeral #1\relax}}

\makeatletter
\providecommand*{\twoheadrightarrowfill@}{%
  \arrowfill@\relbar\relbar\twoheadrightarrow
}
\providecommand*{\twoheadleftarrowfill@}{%
  \arrowfill@\twoheadleftarrow\relbar\relbar
}
\providecommand*{\xtwoheadrightarrow}[2][]{%
  \ext@arrow 0579\twoheadrightarrowfill@{#1}{#2}%
}
\providecommand*{\xtwoheadleftarrow}[2][]{%
  \ext@arrow 5097\twoheadleftarrowfill@{#1}{#2}%
}
\makeatother

\makeatletter
\newcommand\setItemnumber[1]{\setcounter{enum\romannumeral\@enumdepth}{\numexpr#1-1\relax}}
\makeatother

\newcommand\norm[1]{\left\lVert#1\right\rVert}

\makeatletter
\def\ind{\@ifnextchar[{\@with}{\@without}}
\def\@with[#1]#2{\mathrm{Ind}(#1,#2)}
\def\@without#1{\mathrm{Ind}(#1)}
\makeatother

\newcounter{para}[section]

\newcommand\KK[0]{K\! K}

\DeclareMathOperator\End{End}

\DeclareMathOperator\Cl{\mathbb{C}l}

\DeclareMathOperator\dnc{DNC}
\DeclareMathOperator\crit{Crit}

\DeclareMathOperator\critr{RCrit}

\DeclareMathOperator\blup{Blup}

\DeclareMathOperator\graph{graph}

\DeclareMathOperator\Hom{Hom}

\DeclareMathOperator\dom{Dom}

\DeclareMathOperator\codim{Codim}

\newcommand{\R}{\mathbb{R}}

\newcommand{\N}{\mathbb{N}}
\newcommand{\C}{\mathbb{C}}

\begin{document}
\title{Witten deformation using Lie groupoids}
\author{Omar Mohsen}
\date{}
\maketitle
\abstract{We express Witten's deformation of Morse functions using deformation to the normal cone and $C^*$-modules. This allows us to obtain asympotitcs of the `large eigenvalues'. This is then applied to the case of Morse-Bott functions inspired by \cite{MR852155}.

Our methods extend to Morse functions along a foliation. We construct the Witten deformation using any generic function on an arbitrary foliation on a compact manifold and establish the compactness of its resolvent. When the foliation has a holonomy invariant transverse measure we show that our result implies Morse inequialities obtained  by Connes and Fack \cite{MR2276915} in a slightly more general situation.}
\section*{Introduction}In this article, we give an application of deformation groupoids to Witten's deformation of a Morse function. A Morse function $f$ is a real valued smooth function on a compact manifold $M$ with nondegenerate critical points. This is a generic condition by results of Morse. In \cite{MR1451874}, Morse proved the so called Morse inequalities highlighting a relation between the number of critical points of $f$ and the Betti numbers $\dim(H^\cdot (M)).$ 
\begin{theorem}[Morse inequalities]Let $c_i$ be the number of critical point of $f$ of index $i$.
\begin{align*}
 \dim(H^i(M,\R))\leq c_i,\quad \sum_{i=0}^k(-1)^{k-i}\dim(H^i(M,\R))\leq \sum_{i=0}^k(-1)^{k-i}c_i
\end{align*}
\end{theorem}

He did so by studying the level sets $f^{-1}(]-\infty,a])$ and seeing how they change as $a$ passes by a critical value. In \cite{MR683171}, Witten proposed an analytic way to prove Morse inequalities. His method consists of deforming the De Rham operator $d$ to become $d_{t}=e^{-\frac{f}{t}}de^{\frac{f}{t}}$, and then studying the associated Laplacian $\Delta_t=(d_t+d^*_t)^2$. Since the operator $d_t$ is conjugate to $d$, it follows that $\ker(\Delta_t)$ is isomorphic to $\ker(\Delta)$. Hence by Hodge theory, $\dim(\ker(\Delta^i_t))=\dim(H^i(M,\R))$ for all $t>0$, where $\Delta^i_t$ denotes the deformed Laplacian acting on forms of degree $i$. He then proves that, as $t\to 0^+$, the spectrum $sp(\Delta_t^i)$ gets separated into two parts, the first part is finite and consists of an eigenvalue for each critical point of $f$ of index $i$, and the second part consists of eigenvalues which converge to $+\infty$. Morse inequalities are then corollaries of this decomposition. 
In this article, we give another point of view of Witten deformation, a global one which avoids local coordinates. By applying the deformation to the normal cone construction, one obtains  a smooth manifold whose underlying set is equal to $$\dnc(M,\crit(f))=M\times ]0,1]\sqcup_{a\in \crit(f)}T_aM\times\{0\}.$$ The deformed Laplacian $t^2\Delta_t$ with a slight normalisation acts on the submanifold $M\times \{t\}\subseteq \dnc(M,\crit(f)).$ We show that this operator can be glued smoothly to the direct sum of Harmonic oscillators at every critical point acting on  $\sqcup_{a\in \crit(f)}T_aM\times\{0\}\subseteq \dnc(M,\crit(f)).$ The glued operator will be called the global Witten deformation.

 The key tool that is used to be able to express the smoothness of the global Witten deformation is that of Lie groupoids. The theory of Lie groupoids in connection with pseudo-differential operators, $C^*$-algebras, and index theory started in the work of Connes \cite{MR679730,MR775126,MR823176,MR1303779,MR548112}, and was extensively developed. See for example \cite{MR3669118,MR1467076,MR1687747,MR2227132}. The natural projection $\pi_\R:\dnc(M,\crit(f))\to \R$ is a submersion, hence the fibers define a trivial foliation. The Lie groupoid that is used to study the global Witten deformation is the Lie groupoid associated to this foliation (see \cite{MR679730} for the general definition).

Once this operator is constructed, we show that it is a \textit{ regular operator} in the sense of Baaj and Woronowicz, see \cite{MR1325694}. We recall that regularity amounts to saying that continuous functional calculus can be applied to the global Witten deformation.

Furthermore by adapting the classical theorem which says that the De Rham operator on a complete manifold is locally elliptic to the context of Lie groupoids, we show that the global Witten deformation has \textit{compact resolvent} in the sense of Kasparov for $C^*$-modules \cite{MR918241}. Note that the manifold $\dnc(M,\crit(f))$ is not compact.

The compactness of the resolvent of the global Witten deformation gives immediately (see \cref{cont spectrum}) the asymptotics of the eigenvalues of the deformed Laplacian. More precisely, we obtain \begin{theorem}\label{Witten theorem introduction}
 
Let $$\lambda^p_1(t)\leq \lambda^p_2(t)\cdots$$ denote the spectrum of $\Delta_t^p$, then for every $i\in \N$, $$\lim_{t\to 0^+}t\lambda^p_i(t)=\lambda_i^p(0),$$where $\lambda_i^p(0)$ is the $i$'th eigenvalue of harmonic oscillator $$\Delta_0^p:=\bigoplus_{a\in \crit(f)}\left(d+d^*+c(d^2_a(f))\right)^2:\bigoplus_{a\in \crit(f)}L^2(\Lambda^p_\C T_aM)\to\bigoplus_{a\in \crit(f)} L^2(\Lambda^p_\C T_aM),$$ where $L^2(\Lambda^p_\C T_aM)$ is the set of all $L^2$ functions from $T_aM$ to $\Lambda^p_\C T_aM$, $d^2_af$ is the associated $1$-differential form on $T_aM$, and $c$ is the Clifford multiplication.\end{theorem}

Moreover, we extend the above construction and results to the case of Morse-Bott functions, inspired by Bismut's article \cite{MR852155}.

\bigskip

Connes and Fack \cite{MR2276915}, generalised Witten's approach and Morse inequalities for foliations equipped with a holonomy invariant transverse measure. \begin{theorem}
[Connes-Fack Morse inequalities]Let $F$ be a regular foliation on a compact manifold, $\nu$ a holonomy invariant transverse measure, $f$ a nice function (see \cref{nice dfn}), $c_i=\nu(\crit_\mathcal{F}^i(f))$ the $\nu$-measure of longitudinally Morse critical points of $f$ of index $i$, and $\beta_i$ the $\nu$-dimension of the longitudinal De Rham cohomology. One has $$\beta_i\leq c_i,\quad   \sum_{i=0}^k(-1)^{k-i}\beta_i\leq\sum_{i=0}^k(-1)^{k-i}c_i.$$
\end{theorem}
The numbers $\beta_i, c_i$ in the previous theorem are Von Neumann dimensions which are usually real numbers.

\smallskip

We extend as well our results to longitudinal Witten deformation along an arbitrary regular foliations, not necessarily equipped with a holonomy invariant transverse measure. More precisely, let $F$ be a foliation on a compact manifold $M$, $f:M\to \R$ a smooth function such that the longitudinal derivative $d_Ff:M\to F^*$ is transverse to the zero section. This is a generic condition by Thom's transversality theorem. We denote by $\crit_F(f):=\{x\in M:d_Ff(x)=0\}$ the manifold of longitudinal critical points. 

\smallskip

Notice that the points at which $\crit_F(f)$ is transverse to $F$ are the points at which $f$ is longitudinally Morse. In general, a foliation has no compact transversal. Hence $f$ will usually have longitudinal critical points of non Morse type. As we will see below, this adds difficulty to the study of Witten Laplacian.

\smallskip

As in the classical case one could define the deformed longitudinal Laplacian $$(e^{-\frac{f}{t}}d_Fe^{\frac{f}{t}}+e^{\frac{f}{t}}d^*_Fe^{-\frac{f}{t}})^2.$$ which now acts on the $C^*$-algebra of the Lie groupoid associated to the foliation that will be denoted by $\mathcal{G}(M,F)$. The advantage of this is that the deformed Laplacian being longitudinally elliptic has compact resolvent in Kasparov's sense.

\smallskip

Using the recent observation of Debord and Skandalis \cite{Debord:2017aa}, we construct a Lie groupoid $$\dnc(\mathcal{G}(M,F),\crit_F(f)):=\mathcal{G}(M,F)\times \R^*\sqcup \mathcal{N}^{\mathcal{G}(M,F)}_{\crit_F(f)}\times\{0\}.$$

\smallskip

We then construct a global longitudinal Witten deformation which acts on the $C^*$-algebra of the Lie groupoid $\dnc(\mathcal{G}(M,F),\crit_F(f))$ which on $\mathcal{G}(M,F)\times \{t\}$ is equal to $$t^2(e^{-\frac{f}{t}}d_Fe^{\frac{f}{t}}+e^{\frac{f}{t}}d^*_Fe^{-\frac{f}{t}})^2$$ and which is equal to a `Harmonic oscillator' on $\mathcal{N}^{\mathcal{G}(M,F)}_{\crit_F(f)}\times\{0\}$ at the critical points. The Harmonic oscillator at a critical point $x$ depends on the dimension of $T_x\crit_F(f)\cap F_x$. If $x$ is Morse (the intersection is trivial), then the Harmonic oscillator is the same as in the classical case. If not, it is slightly more complicated.
 
\smallskip

The methods that we used to prove in the classical case that the global Witten deformation is regular with compact resolvent can be used in the case of foliations. In \cref{Wittens thm foliations}, we show that the global longitudinal Witten deformation is regular and has compact resolvent as an operator on a $C^*$-module. In fact in \Cref{Basic lemma}, we prove a general result that applies to any Lie groupoid from which we deduce both \cref{Witten theorem introduction} and \cref{Wittens thm foliations}.

\bigskip

To obtain quantitative corollaries of \cref{Wittens thm foliations} and recover Connes-Fack inequalities, we assume that the foliation has a holonomy invariant transverse measure and some extra genericity assumption on $f$, see \cref{nice dfn}.

\smallskip

In \cite{MR548112}, Connes shows that  a holonomy invariant transverse measure defines a trace on the $C^*$-algebra of the foliation. We show that this trace extends to a $C([0,1])$-valued trace on $C^*(\dnc(\mathcal{G}(M,F),\crit_F(f)))$. This $C([0,1])$-trace together with  \cref{Wittens thm foliations} imply Connes-Fack Morse inequalities in a slightly more general situation. 

\smallskip

More precisely, we don't need to assume that the transverse measure is absolutely continuous with respect to the Lebesgue measure. In fact in \cref{continuity trace lemma}, we show a closely related assumption that is pivotal in the construction of the $C([0,1])$-valued trace on $C^*\dnc(\mathcal{G}(M,F),\crit_F(f))$ (see \cref{approx fun}). Such condition is valid for a generic function $f$. We finally show that if $f$ doesn't satisfy the mentioned above condition then the Morse inequalities for $f$ are still true by approximating $f$ by another function which satisfies the required conditions and whose critical points aren't far from those of $f$.

Finally let us remark that all our constructions and results work equally well if the Morse function is replaced by a Novikov $1$-form in both the classical and the foliation case.

\bigskip

This paper is organised as follows :

\smallskip

 In section 1, we recall the notion of the deformation to the normal cone following \cite{Debord:2017aa}.

\smallskip

In section 2, we establish regularity and compactness of the resolvent of an abstract Witten deformation on a Lie groupoid.

\smallskip

In section 3, we show that the results of section 2 admits as a corollary Witten's deformation in the classical case, as well as in the case of Morse-Bott functions like in Bismut's article \cite{MR852155}.

\smallskip

In section 4, we apply the results of section 2 to the case of foliated manifolds.

\smallskip

In section 5, Quantitative corollaries of section 4 are obtained under the additional hypothesis of the existence of holonomy invariant measure. In this section we use an independent result on continuity of $1$-parameter family of traces on $C^*$-algebras which is proven in section 6.

\section*{Acknowledgements}
I would like to thank my Ph.D advisor G. Skandalis for the numerous suggestions and discussions. This work was supported by grants from Région Ile-de-France. I would like to also thank C. Viterbo and the Referee for their helpful remarks.
\section{Deformation to the normal cone}\label{dnc section}
In this section, we recall the deformation to the normal cone construction and its relation to Lie groupoids following \cite{Debord:2017aa}. \textit{The deformation to the normal cone}\index{deformation to the normal cone} of a manifold $M$ along a submanifold $V$ is a manifold whose underlying set is $$\dnc(M,V):=M\times ]0,1]\sqcup N^M_V\times\{0\},$$ where $N^M_V$ is the normal bundle of $V$ inside $M$. The smooth structure is defined by covering $\dnc(M,V)$ with two open sets; the first is $M\times ]0,1]$ and the second is $\phi(N^M_V)\times]0,1]\sqcup N^M_V\times\{0\}$ where $\phi:N^M_V\to M$ is a tubular embedding.\footnote{To simplify the exposition, we will always assume that tubular neighbourhoods are diffeomorphisms on $N^M_V$.} The smooth structure on $\phi(N^M_V)\times]0,1]\sqcup N^M_V\times\{0\}$ is given by declaring the following map a diffeomorphism \begin{align*}
&\tilde{\phi}:N^M_V\times [0,1]\to \phi(N^M_V)\times]0,1]\sqcup N^M_V\times\{0\}\\
 &\tilde{\phi}(x,X,t)=(\phi(x,tX),t)\in M\times ]0,1],\quad t\neq 0\\
 &\tilde{\phi}(x,X,0)=(x,X,0)\in N^M_V\times\{0\}.
\end{align*}
This defines a smooth structure. Independence of $\phi$ follows by noticing that the following functions are smooth functions that generate the smooth structure:\begin{enumerate}
\item the function\begin{align*}
(\pi_M,\pi_\R):\dnc(M,V)&\to M\times \R\\
(x,t)&\to (x,t),\quad t\neq 0\\
(x,X,0)&\to (x,0)
\end{align*}
\item Let $f\in C^\infty(M)$ be a smooth function which vanishes on $V$. Therefore $df:N^M_V\to \R$ is well defined. The following function is smooth
\begin{align*}
\dnc(f):\dnc(M,V)&\to \R\\
(x,t)&\to \frac{f(x)}{t},\quad t\neq 0\\
(x,X,0)&\to df_x(X)
\end{align*}
\end{enumerate}

\begin{prop}[Functoriality of DNC]\label{Functoriality of DNC}
Let  $M,M'$ be smooth manifolds, $V\subseteq M$, $V'\subseteq M'$ submanifolds, $f:M\to M'$ a smooth map such that $f(V)\subseteq V'$. Then the map defined by 
\begin{align*}
\dnc(M,V)&\to \dnc(M',V')\\
(x,t)&\to (f(x),t),\quad t\neq 0\\
(x,X,0)&\to (f(x),df_x(X),0)
\end{align*}
is a smooth map that will be denoted by $\dnc(f)$. 
 Furthermore the map  $\dnc(f)$ is \begin{itemize}
\item  a submersion if and only if $f$ is a submersion and $f_{|V}:V\to V'$ is also a submersion.
 \item an immersion if and only if $f$ is an immersion and for every $v\in V$, $T_vV=df_v^{-1}(TV').$
 \end{itemize}
\end{prop}
The map \begin{align*}
 N^M_V\to N^{M'}_{V'},\quad (x,X)&\to (f(x),df_x(X)) 
\end{align*} will be denoted by $Nf$.

\begin{prop}\label{DNC fibered product}
Let $M_1, M_2, M$ be manifolds, $V_i\subseteq M_i,V\subseteq M$ submanifolds, $f_i:M_i\to M$ smooth maps such that \begin{enumerate}
\item $f_i(V_i)\subseteq V$ for $i\in \{1,2\}$
\item the maps $f_i$ are transverse
\item the maps $f_{i}|V_i:V_i\to V$ are transverse
\end{enumerate} Then \begin{enumerate}
\item\begin{enumerate}[label=(\alph*)]
\item the maps $Nf_i:N^{M_i}_{V_i}\to N^M_V$ are transverse.
\item the natural map $$N^{M_1\times_MM_2}_{V_1\times_VV_2}\to N^{M_1}_{V_1}\times_{N^{M}_{V}}N^{M_1}_{V_2}$$ is a diffeomorphism.

\end{enumerate} 
Similarly for $\dnc$, we have \item\begin{enumerate}[label=(\alph*)]
 \item the maps $\dnc(f_i):\dnc(M_i,V_i)\to \dnc(M,V)$ are transverse.
 \item the natural map $$\dnc(M_1\times_MM_2,V_1\times_VV_2)\to \dnc(M_1,V_1)\times_{\dnc(M,V)}\dnc(M_2,V_2)$$ is a diffeomorphism.
\end{enumerate}

\end{enumerate}
\end{prop}

In the next proposition and for the rest of this article we use the notion of Lie groupoids, $\mathcal{VB}$-groupoids  and invariant operators on Lie groupoids.  We refer the reader to \cite{MR584266,MR896907} for more details on Lie groupoids and their $C^*$-algebras and to \cite{MR941624,MR896907} for more details on $\mathcal{VB}$-groupoids. There exists also a recent survey on pseudo-differential operators on Lie groupoids \cite{ClaireGeorges}. Let $G\rightrightarrows G^0$ be a Lie groupoid. We use the notation $s$ and $r$ for the source and range and for any $x\in G^0$, we use $G_x:=s^{-1}(\{x\})$. We denote the Lie algebroid of $G$ by $\mathfrak{A}G$. We also denote the anchor map by $\natural$.

\begin{theorem}\label{Construction of dnc grouopids}Let $G$ be a Lie groupoid, $H$ a Lie subgroupoid. Then
\begin{enumerate}
\item the space $N^G_H\rightrightarrows N^{G^0}_{H^0}$ is a Lie groupoid whose structure maps are $Ns$, $Nr$ and whose Lie algebroid is equal to $N^{\mathfrak{A}G}_{\mathfrak{A}H}.$ Furthermore, $N^G_H$ is a $\mathcal{VB}$-groupoid over $H$.
\item  the manifold $\dnc(G,H)\rightrightarrows\dnc(G^0,H^0)$ is a Lie groupoid whose structure maps are $\dnc(s)$, $\dnc(r)$ and Lie algebroid is equal to $\dnc(\mathfrak{A}G,\mathfrak{A}H)$.
\end{enumerate}
\end{theorem}
\begin{proof}
Statements $1$ and $2$ are direct consequences of propositions \ref{Functoriality of DNC} and  \ref{DNC fibered product}.
\end{proof}
From now on, for a Lie groupoid $G$ and a Lie subgroupoid $H$, we will use $\mathcal{N}^G_H$ to denote the space $N^G_H$ equipped with the structure of a Lie groupoid given by \Cref{Construction of dnc grouopids}.
\begin{rem}\label{sections of dnc algebr}\label{trivializing dnc Rn}\label{metric dnc E}
 Let $E\to M$ be a vector bundle, $V\subseteq M$ a submanifold, $F\to V$ a subbundle of the restriction of $E$ to $V$. By \Cref{Construction of dnc grouopids}, the space $\dnc(E,F)$ is a vector bundle over $\dnc(M,V).$ Since a section of $\dnc(E,F)$ is determined by its values on the dense set $M\times ]0,1]$. It follows that $$\Gamma(\dnc(E,F))=\{X\in \Gamma(E\times \R):X_{|V\times\{0\}}\in \Gamma(F)\},$$ where $\Gamma$ denotes the set of global smooth sections. 
 
 In the particular case where $F$ is the zero bundle, it is clear that by dividing by $t$, we have an isomorphism from $\dnc(E,V)$ to $\pi_M^*E$ where $\pi_M:\dnc(M,V)\to M$ is the projection map. It follows that to a Euclidean metric on $E$, one associates canonically a Euclidean metric on $\dnc(E,V).$ 
 
 Moreover, the vector bundles $\dnc(E,V)^*$ and $\dnc(E^*,V)$ are canonically isomorphic by the isomorphism \begin{align*}
\dnc(E^*,V)\to \dnc(E,V)^*\\
\alpha \to \left(e\to \frac{1}{t^2}\alpha(e)\right) \;\text{for}\; t\neq 0.  
\end{align*}
\end{rem}
\begin{examps}\label{examps dnc}\begin{enumerate}
\item If $M$ is a smooth manifold, then $$\dnc(M\times M,M)=M\times M\times ]0,1]\sqcup TM\times \{0\}\rightrightarrows M\times \R$$ is the tangent groupoid constructed by A. Connes \cite{MR1303779}. The product law is given by $$(x,y,t)\cdot (y,z,t)=(x,z,t),\; (x,X,0)\cdot (x,Y,0)=(x,X+Y,0).$$

\item Let $L\subseteq G^0$ be a submanifold. Since $N^G_{L}$ is equal to $N^{G^0}_L\oplus \ker(ds)_{|L}$. It follows that the groupoid $\mathcal{N}^G_{L}\rightrightarrows N^{G^0}_L$ is equal to $$\{(X,A,Y):X,Y\in N^{G^0}_L,A\in \mathfrak{A}G, Y-X=\natural(A)\mod TL\},$$with the structural maps \begin{align*}
 s(X,A,Y)=Y,\, r(X,A,Y)=X,\; (X,A,Y)\cdot (Y,B,Z)=(X,A+B,Z).
\end{align*}
\end{enumerate}
\end{examps}
\begin{rem} \label{base dnc rem}It is clear that in the definition of $\dnc$, the manifold $[0,1]$ could be replaced by $\R$ or equally well by $[-1,1]$. \end{rem}In the rest of this section, we will use $\dnc_\R(M,V)=M\times\R^*\sqcup N^M_V\times \{0\}$.

The group $\R^*$ acts on the manifold $\dnc_\R(M,V)$ by the action $$s\cdot (x,t)=(x,ts),\quad s\cdot (x,X,0)=(x,\frac{X}{s},0).$$ This action is free and proper on the open subset $\dnc_\R(M,V)\backslash V\times \R$. The quotient is the classical blowup in differential geometry and is denoted by $\blup(M,V)$.

If $f:M\to M'$ is a smooth map, $V\subseteq M$ and $V'\subseteq M'$ submanifolds such that $f(V)\subseteq V'$. The map $\dnc(f):\dnc_\R(M,V)\to \dnc_\R(M',V')$ is $\R^*$-invariant but it doesn't always descend to the blowup. It only descends when the map $\dnc(f)$ is restricted to the $\R^*$-invariant subset $\dnc(f)^{-1}(V'\times \R)^c$. The quotient $\dnc(f)^{-1}(V'\times \R)^c/\R^*$ is denoted by $\blup_f(M,V)$. Here $^c$ denotes the complement set. It follows that one has a smooth map $$\blup(f):\blup_f(M,V)\to \blup(M',V').$$

The following theorem follows from the naturality construction of $\blup(f)$ and is proved in \cite{Debord:2017aa}. \begin{theorem}[\cite{Debord:2017aa}]\label{DS blowup}
Let $H\subseteq G$ be a Lie subgroupoid. The intersection $$\blup_{r,s}(G,H):=\blup_r(G,H)\cap \blup_s(G,H)\rightrightarrows \blup(G^0,H^0)$$ is a Lie groupoid with structure maps $\blup(r)$ and $\blup(s)$. Its Lie algebroid is $\blup_{\pi}(\mathfrak{A}G,\mathfrak{A}H)$, where $\pi:\mathfrak{A}G\to G^0$ is the projection map.
\end{theorem}
\section{A Preliminary result}\label{Basic lemma}
We will deduce the properties of Witten's deformation ultimately using the following simple proposition.

\begin{prop}\label{lem manifolds Witten}
Let $W$ be a complete Riemannian manifold, $\#:T^*W\to TW$ the musical isomorphism given by the Riemannian metric, $\alpha$ a $1$-form on $W$ such that \begin{enumerate}
\item the form $d\alpha$ is bounded
\item the section of $\End(\Lambda_\C T^*W)$ given by $\mathcal{L}_{\alpha^{\#}}+\mathcal{L}_{\alpha^{\#}}^*$ is bounded
\item $\norm{\alpha}$ is a proper function,
\end{enumerate}
then the operator $d+d^*+c(\alpha)$ acting on $L^2(\Lambda_\C T^*W)$ is a self-adjoint elliptic operator with compact resolvent, where $L^2(\Lambda_\C T^*W)$ is the Hilbert space of $L^2$ sections of $\Lambda_\C T^*W$.
\end{prop}
Here $\norm{\alpha}$ is the function on $W$ which sends a point $x$ to $\norm{\alpha_x}.$ Similarly for $d\alpha$ and $d^*\alpha$.

 In fact we will need the Lie groupoid version of \Cref{lem manifolds Witten}. Before stating the extension to Lie groupoids we will need some classical results for Lie groupoids. We outline the proof of each for the reader's convenience.
 
 Let $g$ be a Euclidean metric on the bundle $\mathfrak{A}G\to G^0$. For every $\gamma\in G$, one has the isomorphism $$T_\gamma G_{s(\gamma)}\xrightarrow{d_{\gamma}R_{\gamma^{-1}}}T_{r(\gamma)}G_{r(\gamma)}=\mathfrak{A}_{r(\gamma)}G,$$ where $R_{\gamma^{-1}}$ denote right multiplication by $\gamma^{-1}$. It follows that $g$ defines a Riemannian metric on $G_x$ for every $x\in G^0$.
  The metric $g$ is called complete if the induced metric on every $G_x$ is complete.
  \begin{prop}\label{Euclidean metric groupoids}There exists a Euclidean metric $g$ on $\mathfrak{A}G$ such that for every $x\in G^0$, the induced Riemannian metric on $G_{x}$ is complete.
\end{prop}
\begin{proof}
Let $g$ be any Euclidean metric on $\mathfrak{A}G$, and let $h:G^0\to ]0,+\infty[$ be a smooth function such that if $x\in G^0$, then the ball in $G_x$ of radius $h(x)$ with center $x$ is relatively compact. It is straightforward to verify that the euclidean metric $\frac{1}{h^2}g$ is complete. See \cite{MR0133785} for more details.
\end{proof}
  \begin{rem}\label{rem compete metric compact}A consequence of the \cref{Euclidean metric groupoids} is that if $G^0$ is compact, then every Euclidean metric on $\mathfrak{A}G$ is complete.
  \end{rem}
Extension of Chernoff's theorem  \cite{MR0369890} to Lie groupoids in the case where $G^0$ is compact was done in \cite{MR2227132} and in the case of foliation Lie groupoids in \cite{MR1050489}. We present here an elementary proof for the general case.
  \begin{prop}\label{D regular groupoid}
Let $G\rightrightarrows G^0$ be a Lie groupoid, $g$ a complete Euclidean metric on $\mathfrak{A}G$, $E\to G^0$ a Hermitian vector bundle, $D$ a symmetric first order $G$-invariant differential operator on $G$ acting on $r^*E$, $c:G^0\to \R$ the function $$c(x)=\sup_{v\in \mathfrak{A}G_x^*:\norm{v}=1}\norm{\sigma(D)(x,v)}.$$ Here $\sigma$ is the principal symbol. If $c$ is bounded above, then the closure of $D$ is a regular self adjoint operator acting on $C^*(E)$.
\end{prop}
\begin{proof}
Let $f\in\Gamma_c\left(r^*\left(E\otimes |\Lambda|^\frac{1}{2}\mathfrak{A}G\right)\right)$. We recall that throughout this article $|\Lambda|^\alpha$ denotes the bundle of $\alpha$-densities. Consider the differential equation $$\partial_t u(\gamma,t)=iDu(\gamma,t),\; u(\gamma,0)=f(\gamma),\quad (\gamma,t)\in G\times \R.$$ By the classical theory of linear differential equations a unique $C^\infty$ solution to this equation exists locally. By Chernoff's theorem \cite{MR0369890} and our assumptions, we deduce that a solution exists globally on $G_x$ for each $x$. In particular solutions to this equation exist globally on $G$. Furthermore the distribution kernel associated to this equation is proper for each fixed $t$. Let \begin{align*}
 V_t:\Gamma_c\left(r^*\left(E\otimes |\Lambda|^\frac{1}{2}\mathfrak{A}G\right)\right)\to\Gamma_c\left(r^*\left(E\otimes |\Lambda|^\frac{1}{2}\mathfrak{A}G\right)\right),\quad f\to u(\cdot,t)
\end{align*}
be the convolution to the left by the distribution kernel. If $f,g\in \Gamma_c\left(r^*\left(E\otimes |\Lambda|^\frac{1}{2}\mathfrak{A}G\right)\right)$, then $$\frac{d}{dt}\langle V_tf,V_tg\rangle=\langle iDV_tf,V_tg\rangle+\langle V_tf,iDV_tg\rangle=\langle i(D-D^*)V_tf,V_tg\rangle=0.$$
Hence the operators $V_t$ extend to an isometry acting on the $C^*G$-module $C^*E$. This operator is adjointable (and therefore $C^*G$-linear) because of the equation $$\langle \xi ,V_t\eta\rangle=\langle V_{-t}\xi,\eta\rangle,$$ which proves as well that $V_t$ is a unitary in $\mathcal{L}(C^*E).$ The proposition follows then from \Cref{lemm to prove Chrnoff}. \end{proof}
\begin{prop}\label{lemm to prove Chrnoff}Let $S$ be a regular self adjoint operator acting on a $C^*$-module $E$, $V_t=\exp(itS)$, $T:\dom(T)\subseteq E\to E$ a $\C$-linear map with a dense domain. If\begin{enumerate}
\item  $V_t\dom(T)=\dom(T)$
\item  $T\subseteq S$.
\end{enumerate}Then the closure of $\graph(T)$ is equal to $\graph(S)$.
\end{prop}
\begin{proof}
By taking the closure of $T$, we can suppose that $T$ is closed. Let $f\in \mathcal{S}(\R)$ be a Schwartz function. Since $$f(S)=\int_{-\infty}^\infty \hat{f}(t)V_{2\pi t}dt,$$ it follows that $f(S)\dom(T)\subseteq \dom(T)$ and $Tf(S)=f(S)T$. Since $f(S)$ and $Sf(S)$ are bounded operators and $\dom(T)$ is dense, it follows that $\{(f(S)x,Sf(S)x):x\in E\}\subseteq \graph(T).$ 

Let $0\leq f_n\leq 1$ be Schwartz functions such that $f_n\to 1$ uniformly on every compact. It follows that $f_n(S)$ strongly converges to the identity. In particular if $x\in \dom(S)$, then $f_n(S)x\to x$, and $Sf_n(S)x\to Sx$. Hence $(x,Sx)\in \graph(T)$, which implies that $S=T$.
\end{proof}
\begin{prop}[\cite{MR2227132}]\label{local compact reolvent}Under the same hypothesis as \Cref{D regular groupoid}, if furthermore $D$ is an elliptic operator, then for every $f\in C_0(G^0)$ and $g\in C_0(\R)$, the operator $g(D)f$ is compact in the sense of $C^*$-modules.
\end{prop}
\begin{proof}
By a density argument it is enough to prove the proposition for $f\in C_c^\infty(G)$, and $g\in C_c(\R)$. Let $Q$ be a parametrix for $D^2$, that is $D^2Q=1+R$ with $R$ a $G$-pseudo differential operator of order $\leq -1$. The support of $Q$ can be chosen to be a subset of an  arbitrary open neighbourhood of $G^0$. Since $D^2$ is a differential operator, its Schwartz kernel is supported in a subset of $G^0$. In particular the support of $R$ can be choosen as well to be a subset of an arbitrary neighbourhood of $G^0$. We choose the supports of $R$ and $Q$ so that $Qf$ and $Rf$ are $G$-invariant pseudodifferential operators with compact support. It follows from \cite[theorem 18]{MR2227132}, that $Qf$ and $Rf$ extend to compact operators on $C^*(E)$. It follows from the identity \begin{equation*}
(1+D^2)^{-1}f=Qf-(1+D^2)^{-1}Rf+(1+D^2)^{-1}Qf,
\end{equation*}
that $(1+D^2)^{-1}f$ is compact. Since $g(x)(1+x^2)$ is bounded, it follows that $g(D)f$ is compact as well.\end{proof}
We recall the following extension of a classical result to Lie groupoids. If $X\in \Gamma^\infty_c(\mathfrak{A}G)$, then the operator $\mathcal{L}_X:\Gamma^\infty_c(\Lambda_\C\mathfrak{A}G^*)\to \Gamma^\infty_c(\Lambda_\C\mathfrak{A}G^*)$ is defined by Cartan's formula $$\mathcal{L}_X=di_X+i_Xd.$$ Here $i_X$ is the interior product acting on differential forms. The operator $\mathcal{L}_X$ is a $G$-operator in the sense of \cite{MR1687747,MR2227132}. We then have \begin{prop}\label{LX+LX* is 0 order alg}The operator $\mathcal{L}_X+\mathcal{L}_X^*$ is $C^\infty(G^0)$-linear (i.e. a $0$-order $G$-differential operator).
\end{prop}
\begin{rem}\label{adjoint of LX rem}Notice that the adjoint of $L_X$ is the adjoint in the sense of $G$-operators. The metric $g$ induces a $G$-invariant Riemannian metric on each $G_x$ for $x\in G^0$ and $X$ induces a $G$-invariant vector field on each $G_x$. The dual $\mathcal{L}_X^*$ is the $G$-invariant dual on each $G_x$.
\end{rem}
\begin{proof}
The operator $\mathcal{L}$ is an ungraded derivation. Therefore, $$\mathcal{L}_X(f\alpha)=(\mathcal{L}_Xf)\alpha+f\mathcal{L}_X\alpha,$$ where $f:G^0\to \R$ is a real valued smooth function and $\alpha\in \Gamma_c^\infty(\Lambda_\C \mathfrak{A}^*G)$. Taking the dual one deduces that $$f\mathcal{L}_X^*(\alpha)=(\mathcal{L}_Xf)\alpha+\mathcal{L}_X^*(f\alpha).$$  Therefore \begin{equation*}
 \mathcal{L}_X(f\alpha)+\mathcal{L}_X^*(f\alpha)=f\left(\mathcal{L}_X(\alpha)+\mathcal{L}_X^*(\alpha)\right).\qedhere
\end{equation*}
\end{proof}
\begin{theorem}\label{keyprop witten}Let $G$ be a Lie groupoid, $g$ a complete euclidean metric on $\mathfrak{A}G$, $\#:\mathfrak{A}G^*\to \mathfrak{A}G$ the musical isomorphism given by $g$, $\alpha\in \Gamma^\infty(\mathfrak{A}G^*)$. If \begin{enumerate}
\item the form $d\alpha$ is bounded
\item the section of $\End(\Lambda_\C \mathfrak{A}G^*)$ given by $\mathcal{L}_{\alpha^{\#}}+\mathcal{L}_{\alpha^{\#}}^*$ is bounded
\item $\norm{\alpha}:G^0\to \R$ is a proper function,
\end{enumerate}
then the closure of the operator $d+d^*+c(\alpha)$ acting on the $C^*(G)$ Hilbert module $C^*\left(\Lambda_\C\mathfrak{A}G^*\right)$ is a regular self adjoint elliptic operator with compact resolvent in the sense of $C^*$-modules.
\end{theorem}
\begin{rem}Thanks to \cite{MR1435704} (see also \cite{MR715325,Lesch:2018aa}), \cref{keyprop witten} implies that the Kasparov product of $d+d^*$ seen as an element of $\KK^0(\Cl(\mathfrak{A}^*G),\C)$ and $c(\alpha)$ seen as an element of $\KK^0(\C,\Cl(\mathfrak{A}^*G))$ is the operator $d+d^*+c(\alpha)\in KK^0(\C,\C)$, where $\Cl$ denotes the complex Clifford algebra.
\end{rem}
\begin{proof}
Since $$\norm{\sigma\left(d+d^*+c(\alpha)\right)(x,v)}=\norm{\sigma\left(d+d^*\right)(x,v)}=\norm{v}_{g_x},$$ it follows that $d+d^*+c(\alpha)$ and $d+d^*$ are elliptic and from \cref{D regular groupoid} that the closure of $d+d^*+c(\alpha)$ and $d+d^*$ are regular self adjoint operators.

\bigskip

It follows from Cartan's formula that the graded commutator is equal to
\begin{align*}
[d,i_{\alpha^{\#}}]=\mathcal{L}_{\alpha^{\#}}.
\end{align*}

Since  \begin{align*}
[d,c(\alpha)]&=[d,\alpha \wedge]+[d,i_{\alpha^{\#}}]=d\alpha\wedge\cdot + \mathcal{L}_{\alpha^{\#}}.
\end{align*}Hence by the hypotheses of \cref{keyprop witten}\begin{align*}
 [d+d^*,c(\alpha)]=[d,c(\alpha)]+[d,c(\alpha)]^*=d\alpha \wedge \cdot +i_{(d\alpha)^\#}+\mathcal{L}_{\alpha^{\#}}+\mathcal{L}^*_{\alpha^{\#}}
\end{align*} is bounded, where $i_{(d\alpha)^\#}(\cdot)$ is the adjoint of $d\alpha \wedge \cdot$. Therefore the closure of $(d+d^*)^2+c(\alpha)^2=(d+d^*+c(\alpha))^2-[d+d^*,c(\alpha)]$ is a regular self adjoint operator.

\bigskip

By a classical inequality (see \cite{MR1325694}), one has \begin{align*}
& (1+(d+d^*)^2+c(\alpha)^2)^{-1}\leq (1+(d+d^*)^2)^{-1}\\&(1+(d+d^*)^2+c(\alpha)^2)^{-1}\leq (1+(c(\alpha))^2)^{-1}=(1+\norm{\alpha}^2)^{-1}
\end{align*}
It follows from \cite[proposition 1.4.5]{MR548006} that there exists $a,b\in \mathcal{L}(C^*\Lambda_\C\mathfrak{A}G^*)$ such that $$(1+(d+d^*)^2+c(\alpha)^2)^{-\frac{1}{2}}=a(1+(d+d^*)^2)^{-\frac{1}{4}},\quad (1+(d+d^*)^2+c(\alpha)^2)^{-\frac{1}{2}}=(1+\norm{\alpha}^2)^{-\frac{1}{4}}b.$$ Hence $$(1+(d+d^*)^2+c(\alpha)^2)^{-1}=a(1+(d+d^*)^2)^{-\frac{1}{4}}(1+\norm{\alpha}^2)^{-\frac{1}{4}}b.$$ Since by our assumptions $(1+\norm{\alpha}^2)^{-\frac{1}{4}}\in C_0(G^0)$. It follows that $$(1+(d+d^*)^2)^{-\frac{1}{4}}(1+\norm{\alpha}^2)^{-\frac{1}{4}}\in \mathcal{K}(C^*\Lambda_\C\mathfrak{A}G^*)$$Hence $(1+(d+d^*)^2+c(\alpha)^2)^{-1}$ is compact as well.

\bigskip

Since $[d+d^*,c(\alpha)]$ is bounded, and \begin{align*}
 &\left(1+\left(d+d^*+c\left(\alpha\right)\right)^2\right)^{-1}\\&=\left(1-\left(1+\left(d+d^*+c(\alpha)\right)^2\right)^{-1}[d+d^*,c(\alpha)]\right)\left(1+(d+d^*)^2+c(\alpha)^2\right)^{-1},
\end{align*}
it follows that $(1+(d+d^*+c(\alpha))^2)^{-1}$ is compact.
\end{proof}
\subsection{Completion}\label{sect comple}
 \begin{prop}\label{comple prop}
Let $G\rightrightarrows G^0$ be a Lie groupoid, $g$ a complete Euclidean metric on $\mathfrak{A}G$ and $\alpha\in \Gamma(\mathfrak{A}^*G)$ a $1$-form.

Let $U$ be a saturated open subset of $G^0$, $V:=U^c\subseteq G^0$ its complement, $h:G^0\to [0,+\infty[$ a bounded smooth positive function such that  \begin{enumerate}
\item for any $x\in V$, $\alpha(x)\neq 0$.
\item $h^{-1}(0)=V$ 
\item $\norm{\frac{dh}{h}}_g\norm{\alpha}_g$ is bounded on $U$, where $dh\in \Gamma(\mathfrak{A}^*G)$ is the composition of the De Rham derivative with the anchor map.
\end{enumerate}
Then \begin{enumerate}
\item if the pair $g$ and $\alpha$ satisfy the hypotheses of \cref{keyprop witten} on the Lie groupoid $G$, then pair $\frac{g}{h}$ and $\frac{\alpha}{h}$ satisfy the hypotheses of \cref{keyprop witten} on the Lie groupoid $G_{|U}:=s^{-1}(U)=r^{-1}(U)$.
\item If for some $\epsilon$, the set $h^{-1}([0,\epsilon])$ is compact, then the converse of $1$ holds.
\end{enumerate}
\end{prop}
\begin{proof}
The metric $\frac{g}{h}$ is complete because $h$ is bounded. This follows from the inclusion for any $x\in U$, $r>0$ $$B_\frac{g}{h}(x,r)\subseteq B_g(x,\frac{r}{\norm{h}_{\infty}}),$$
where $B_g(x,r)$ and $B_{\frac{g}{h}}(x,r)$ denotes the ball of radius $r$ with center $x$  in $G_x$ with respect to $g$ and $\frac{g}{h}.$

Let us verify the equivalence \begin{enumerate}
\item First one has $$\norm{d(\frac{\alpha}{h})}_{\frac{g}{h}}=h\norm{d(\frac{\alpha}{h})}_{g}=\norm{d\alpha -\frac{dh}{h}\alpha}_g.$$ Since $\norm{\frac{dh}{h}}_g\norm{\alpha}_g$ is bounded, it follows that $\norm{d(\frac{\alpha}{h})}_{\frac{g}{h}}$ is bounded if and only if $\norm{d\alpha}_g$ is bounded.
\item Let $X=\alpha^{\#_{g}}=\left(\frac{\alpha}{h}\right)^{\#_{\frac{g}{h}}}.$ The equivalence of the second condition follows from lemma \ref{lem 2} and the following inequality $$|\frac{dh(X)}{h}|\leq \norm{\frac{dh}{h}}_g\norm{X}_g=\norm{\frac{dh}{h}}_g\norm{\alpha}_g$$\begin{lem}\label{lem 2}Let $G\rightrightarrows G^0$ be  a Lie groupoid, $h:G^0\to ]0,+\infty[$ a smooth function, $g$ a Euclidean metric on $\mathfrak{A}G$ and $X\in \Gamma(\mathfrak{A}G^0)$. If the function $\frac{dh(X)}{h}$ is bounded, then the section $\mathcal{L}_X+\mathcal{L}_X^{*_g}$ is bounded if and only if the section $\mathcal{L}_X+\mathcal{L}_X^{*_{\frac{g}{h}}}$ is bounded.
\end{lem}
\begin{proof}
Recall \cref{adjoint of LX rem}, and let $x\in G^0$, $\tilde{h}:=h\circ r:G_x\to \R$, $\alpha,\beta\in \Gamma(\Lambda^k_\C T G_x^*)$ \begin{align*}
\int_{G_x}\langle\mathcal{L}_X \alpha,\beta\rangle_{\frac{g}{h}}d vol_{\frac{g}{h}}&=\int_{G_x}\tilde{h}^{k}\langle \mathcal{L}_X \alpha,\beta\rangle_g \tilde{h}^{-\frac{n}{2}}d vol_{g}\\&=\int_{G_x}\langle \alpha,\mathcal{L}_X^*\left(\tilde{h}^{k-\frac{n}{2}}\beta\right)\rangle_g dvol_g\\&=\int_{G_x}\langle\alpha,\tilde{h}^{\frac{n}{2}-k}\mathcal{L}_X^*\left(\tilde{h}^{k-\frac{n}{2}}\beta\right)\rangle_{\frac{g}{h}}dvol_{\frac{g}{h}}.
\end{align*}
It follows that $$\mathcal{L}_{X}^{*_{\frac{g}{h}}}=h^{\frac{n}{2}-k}\mathcal{L}_X^{*_g}h^{k-\frac{n}{2}}$$ on $\Gamma(\Lambda_\C^k\mathfrak{A}G).$

Notice that since the metrics $g$ and $\frac{g}{h}$ are conformal, one has $$\norm{\mathcal{L}_X+\mathcal{L}_X^{*_{\frac{g}{h}}}}_{g}=\norm{\mathcal{L}_X+\mathcal{L}_X^{*_{\frac{g}{h}}}}_{\frac{g}{h}}\in C^\infty(G^0).$$ Hence using the metric $g$ or $\frac{g}{h}$ is of no consequence on the boundness. It follows that to prove the proposition it suffices to prove that $\norm{\mathcal{L}_X^{*_g}-\mathcal{L}_{X}^{*_{\frac{g}{h}}}}_g$ is bounded. One has \begin{align*}
\norm{\mathcal{L}_X^{*_g}-\mathcal{L}_{X}^{*_{\frac{g}{h}}}}_g&=\norm{\mathcal{L}_X^{*_g}-h^{\frac{n}{2}-k}\mathcal{L}_X^{*_g}h^{k-\frac{n}{2}}}_g\\&=\norm{\mathcal{L}_X-h^{k-\frac{n}{2}}\mathcal{L}_Xh^{\frac{n}{2}-k}}_g\\&=|h^{k-\frac{n}{2}}\natural{X}(h^{\frac{n}{2}-k})|=|(\frac{n}{2}-k-1)\frac{dh(X)}{h}|.
\end{align*}
The lemma then follows.
\end{proof}
\item The last equivalence follows from the identity $$\norm{\frac{\alpha}{h}}_{\frac{g}{h}}^2=\frac{1}{h}\norm{\alpha}_g^2,$$and the fact that $\norm{\alpha}_g$ is non zero on $V$.
\end{enumerate}
\end{proof}
 \begin{rem}\label{rem V dim 1}suppose that $V:=U^c$ is a smooth submanifold of $G^0$. This implies that for any $x\in V$,  \begin{equation}
 \natural(\mathfrak{A}_xG)\subseteq T_xV
\end{equation}
If furthermore $V$ is of codimension $1$, and if $h\in C^\infty(G^0)$ is a smooth function such that locally if $G^0=\R^n$, $V=\R^{n-1}$, $h(x)=x_n^2$ for some local coordinates, then the section of $\mathfrak{A}G$ on $U$, $\frac{dh}{h}$ extends to a smooth section on $G^0$. This follows from the fact that $dh$ is zero on $V$.
\end{rem}
\section{Witten deformation}\label{section classical witten deformation}
\subsection{Morse function}\label{Classical Witten subsection 1}
Let $M$ be a closed manifold, $f:M\to \R$ a Morse function. We denote by $\crit(f)$ the set of its critical points (a finite set), and by $\pi_\R:\dnc(M,\crit(f))\to [0,1]$, $\pi_M:\dnc(M,\crit(f))\to M$ the natural projections. By \Cref{Construction of dnc grouopids}, the following is naturally a Lie groupoid \begin{align*}
G&=\dnc(M\times M,\crit(f)\times \crit(f))\\&=M\times M\times ]0,1]\sqcup_{a,b\in \crit(f)}T_aM\times T_bM\times\{0\}\rightrightarrows \dnc(M,\crit(f)),
\end{align*}
whose algebroid is equal to $\dnc(TM,\crit(f)).$ 

Let $g$ be a Riemannian metric on $M$. In \cref{metric dnc E}, on $\mathfrak{A}G=\dnc(TM,\crit(f))$ a Euclidean metric is defined which on $M\times\{t\}$ is equal to $\frac{g}{t^2}$ for $t\neq 0$ and the constant Riemannian metric $g_a$ on $T_aM\times\{0\}.$ This metric is complete by the completeness of the metric $g$ on $M$.

\bigskip

Let $\alpha$ be the $1$-form given by \cref{Functoriality of DNC} $$\alpha=\dnc(df):\dnc(M,\crit(f))\to \dnc(T^*M,\crit(f)).$$ After identifying $\dnc(T^*M,\crit(f))$ with $\dnc(TM,\crit(f))^*=\mathfrak{A}G^*$ (see \Cref{metric dnc E}), the form $\alpha$ is equal to $\frac{df}{t^2}$ on $M\times\{t\}$ for $t\neq 0$ and to $d^2_af$ on $T_aM\times\{0\}$ for $a\in\crit(f).$

\bigskip

Let us verify the condition of \Cref{keyprop witten}. \begin{enumerate}
 \item The form $\alpha$ is clearly closed.
 \item On $M\times\{t\}$, one has $$\alpha^{\#_{\frac{g}{t^2}}}=df^{\#_g},$$ where $\#$ is the musical isomorphism. Hence $\mathcal{L}_{\alpha^{\#}}$ is independent of $t$. Since the Riemannian metric is multiplied by a scalar, it follows that $\mathcal{L}_{\alpha^{\#}}^*$ doesn't depend on $t$ as well. In other words, the section $\mathcal{L}_{\alpha^{\#}}+\mathcal{L}_{\alpha^{\#}}^*$ on $\dnc(M,\crit(f))$ is the pullback of the section $\mathcal{L}_{df^{\#_g}}+\mathcal{L}_{df^{\#_g}}^*$ on $M$ using the projection map $\pi:\dnc(M,\crit(f))\to M$. Hence the norm of the section $\mathcal{L}_{\alpha^{\#}}+\mathcal{L}_{\alpha^{\#}}^*$ is bounded by its boundness on $M$.

\item On $M\times\{t\}$, one has $$\norm{\alpha}_{\frac{g}{t^2}}=\norm{\frac{df}{t^2}}_{\frac{g}{t^2}}=\frac{\norm{df}_g}{t}$$ and on $T_aM\times\{0\}$, $$\norm{\alpha}_{g_a}=\norm{d^2_af}_{g_a}.$$Hence $\norm{\alpha}$ is a proper function on the space $\dnc(M,\crit(f))$.

\end{enumerate}
\begin{cor}\label{cor wtitten compact resolvent M}The operator $d+d^*+c(\alpha)$ acting on the $C([0,1])$ module $C^*(\Lambda_\C\ker(d\pi_\R)^*)$ is a regular self adjoint operator with compact resolvent.
\end{cor}
\begin{proof}
The manifold $\dnc(M,\crit(f))$ gives naturally a Morita equivalence between the Lie groupoid $\dnc(M\times M,\crit(f)\times\crit(f))\rightrightarrows \dnc(M,\crit(f))$ and the trivial Lie groupoid $[0,1]\rightrightarrows [0,1]$. The corollary then follows from \Cref{keyprop witten}.
\end{proof}
\begin{cor}\label{cor witen}Let $d_t=e^{-\frac{f}{t}}de^{\frac{f}{t}}$, $\Delta_t=(d_t+d_t^*)^2$ be the Witten Laplacian acting on $L^2(\Lambda_\C T^*M)$. If $$\lambda^p_1(t)\leq \lambda^p_2(t)\cdots$$ denotes the spectrum of $\Delta_t$  acting on $p$-forms, then the function $$t\to \begin{cases}t\lambda^p_i(t)\quad\text{if}\;t\neq 0\\ \lambda_i^p(0)\quad \text{if} \;t=0\end{cases}$$ is continuous, where $\lambda_i^p(0)$ is the $i$'th eigenvalue of Harmonic osccilator $$\bigoplus_{a\in \crit(f)}\left(d+d^*+c(d^2_a(f))\right)^2:\bigoplus_{a\in \crit(f)}L^2(T_aM,\Lambda_\C^pT_aM)\to\bigoplus_{a\in \crit(f)} L^2(T_aM,\Lambda_\C^pT_aM),$$ where $L^2(T_aM,\Lambda_\C^p T_aM)$ is the set of all $L^2$ functions from $T_aM$ to $\Lambda_\C^p T_aM$, $d^2_af$ is the $1$-differential form on $T_aM$, and $c$ is the Clifford multiplication.
\end{cor}
\begin{proof}
After normalizing the metric $\frac{g}{t^2}$, the operator $(d+d^*+c(\alpha))^2$ on $M\times\{t\}$ is equal to $t^2\Delta_{t^2}.$ The corollary then follows from \Cref{cont spectrum}.\end{proof}\begin{lem}\label{cont spectrum}Let $E$ be a $C[0,1]$ module, $L\in \mathcal{K}(E)$ a compact operator, $\mu_i(t)$ denote the singular value of the operator $|L_t|$. The function $t\to \mu_i(t)$ is continuous.
\end{lem}
\begin{proof} 
By \cite[theorem 2.1]{MR0246142}, one has if $T_1,T_2$ are compact operators, then for every $i\in \N$, $$|\mu_i(T_1)-\mu_i(T_2)|\leq \norm{T_1-T_2}.$$It follows that if $H$ is a Hilbert space and $E=H\otimes C([0,1])$ is a constant $C([0,1])$-module the result follows. For general modules the results follows from Kasparov stabilisation theorem \cite[theorem 13.6.2]{MR1656031}.
\end{proof}
\begin{rem} \Cref{cont spectrum} is false if $L$ is only supposed to be in $\mathcal{L}(E),$ and $L_t$ is compact for every $t$. For example if $E=C_{0}(]0,1])$ and $L$ the identity.
\end{rem}
\bigskip

The calculation of the spectrum of the harmonic oscillator in \Cref{cor witen} is a classical calculation. In particular we have \begin{prop}[{\cite[section V.3]{MR0493419}}]\label{computation spectrum harmonic osc}If $Q$ is a quadratic form on a real euclidean vector space $V$ with signature $(p,q)$ with $p+q=n$ and $$\xi_1\leq \dots\leq \xi_{q}<0< \xi_{q+1}\leq \dots\leq \xi_{n},$$ denote the eigenvalues of $Q$, then the spectrum of $$\left(d+d^*+c(Q)\right)^2:L^2(V,\Lambda_\C^kV^*)\to L^2(V,\Lambda_\C^kV^*)$$ is the weighted set\footnote{the union is with multiplicity} \begin{equation*}
\coprod_{\substack{J\subseteq \{1,\dots,n\}:|J|=k\\
\alpha\in \N^{n}}}\{\sum_{j\in J\Delta\{1,\dots,q\}}|\xi_j|+\sum_{j=1}^{n}\alpha_j|\xi_j|\}.
\end{equation*}
Here $c(Q)$ means the Clifford multiplication by the $1$-form associated to $Q$.
\end{prop}

\begin{cor}[Morse inequalities]If $C_i$ denotes the number of critical points of $f$, then for every $k$, $$\sum_{i=0}^k(-1)^{k-i}C_i\geq \sum_{i=0}^k(-1)^{k-i}\dim H^i(M,\R)$$
\end{cor}
\begin{proof}
Multiplication by $e^f$ is an isomorphism between the complex $(\Omega^*(M),d+df)$ and $(\Omega^*(M),d).$ The corollary follows from Hodge theory and \Cref{cor witen}.
\end{proof}
\begin{rems}\begin{enumerate}
\item It is clear that the above proof verbatim works for Novikov Morse $1$-forms.
\item We could have equally well used the groupoid $\dnc(M\times M,\crit(f)).$ The main difference would be that $C^*(\dnc(M\times M,\crit(f)))$ is not Morita equivalent to $C([0,1])$ which makes the arguments slightly more complicated. One would need to use traces to obtain \cref{cor witen} as done in \cref{qualit result sect}.
\end{enumerate}
\end{rems}
\subsection{Morse-Bott functions}\label{sect Morse Bott}
Let $f:M\to \R$ be a Morse-Bott function. In this section it is simpler to work with $\dnc_{[-1,1]}$ than $\dnc_{[0,1]}$ to avoid difficulties with the boundary, see \cref{base dnc rem}. To simplify the notation we will use $\dnc$ to denote $\dnc_{[-1,1]}$. As in \cref{section classical witten deformation}, the groupoid \begin{align*}
 G&=\dnc(M\times M,\crit(f)\times \crit(f)) \rightrightarrows \dnc(M,\crit(f))\\&=\dnc(M,\crit(f))\times_{[-1,1]}\dnc(M,\crit(f))\rightrightarrows \dnc(M,\crit(f))
\end{align*} will be used to describe a global Witten deformation. 

The Lie algebroid of $G$ is $\dnc(TM,T\crit(f)).$ This vector bundle is isomorphic but not canonically to the vector bundle $\pi^*TM.$ In particular a Riemannian metric on $TM$ doesn't automatically give a Riemannian metric on $\dnc(TM,T\crit(f))$. In the other hand we can use a natural compactification of the space $G^0$ and \cref{sect comple}.

By \cref{DS blowup}, let  $$K=\blup_{r,s}(M\times M\times [-1,1],\crit(f)\times\crit(f)\times \{0\})\rightrightarrows \blup(M\times [-1,1],\crit(f)\times \{0\}).$$

One can immediately see that $$K^0=M\times [-1,1]\backslash \{0\}\sqcup M\backslash \crit(f)\times \{0\}\sqcup \mathbb{P}(N^M_{\crit(f)})\sqcup N^{M}_{\crit(f)}\times \{0\}$$ and \begin{align*}
K&=M\times M\times  [-1,1]\backslash \{0\}\sqcup M\backslash\crit(f)\times M\backslash\crit(f)\times  \{0\}\\&\sqcup (N^M_{\crit(f)}\backslash \{0\}\times N^M_{\crit(f)}\backslash \{0\})/\R^*\sqcup N^M_{\crit(f)}\times N^M_{\crit(f)}\times \{0\}. 
\end{align*} It follows that $\dnc(M,\crit(f))$ can be seen as a saturated open subset of $K^0$. Furthermore $G=K_{|\dnc(M,\crit(f))}.$

The set $$V=K^0\backslash \dnc(M,\crit(f))=M\backslash \crit(f)\times \{0\}\sqcup \mathbb{P}(N^M_{\crit(f)})=\blup(M\,\crit(f))$$ is a submanifold of $K^0$ of codimension $1$. Let $h\in C^\infty(K^0)$ be as in \cref{rem V dim 1}, a smooth non-negative function such that $h^{-1}(0)=V$ and locally $h$ is the square of a distance function to $V$. The choice of $h$ is unique up to multiplication by a positive smooth function on $K^0$. \begin{examp}\label{examp h} Let $(g_i)_{i=1}^{k}$ be a finite family of  smooth functions on $M$ such that \begin{enumerate}
\item $g_i(\crit(f))=0$ and $\cap_i g_i^{-1}(0)=\crit(f)$.
\item for every $x\in \crit(f)$ the linear map $T_xM/T_x\crit(f)\to \R^k$ given by $X\to (dg_1(X),\dots,dg_i(X))$ is injective.
\end{enumerate}The following function is an example of a function $h:K^0\to \R$ $$ h=\begin{cases}h(x,t)= \frac{t^2}{\sum_i g_i(x)^2+t^2}\quad\text{if} \; (x,t)\in M\times \R^*\\h(x,X,0)=\frac{1}{\sum dg_i(X)^2 +1}\quad  \text{if}\; (x,X,0)\in N^M_{\crit(f)}\\0 \quad \text{if not}
\end{cases}$$
\end{examp}

On the manifold $\dnc(M,\crit(f))$, one has the differential form $$\alpha=\dnc(df):\dnc(M,\crit(f))\to \dnc(T^*M,T^{\perp}\crit(f))=\dnc(TM,T\crit(f))^*.$$

\begin{lem}\label{lem h2}The form $h\alpha$ extends to a section of $\mathfrak{A}K^*$ furthermore it is nowhere zero on $V$.
\end{lem}
\begin{proof}
This clearly doesn't depend on the choice of $h$. So it is enough to prove it for $h$ given in \cref{examp h}.

Let $U_i$ be an open cover of $M$, such that $f$ is constant on $U_i\cap \crit(f)$ for every $i$. It follows that $\blup(U_i\times \R,U_i\cap \crit(f)\times \{0\})$ and $\dnc(U_i,U_i\cap\crit(f))$ is an open cover of $\blup(M\times \R,\crit(f)\times \{0\})$ and $\dnc(M,\crit(f))$ respectively. Hence to prove the lemma it is enough to suppose that $f$ is constant on $\crit(f)$. Let $c$ denote $f(\crit(f))$. The following function $\phi:\dnc(M,\crit(f))\to \R$ is well defined and smooth by results in \cref{dnc section}$$\phi=\begin{cases}\phi(x,t)=\frac{f(x)-c}{t^2}\quad \text{if}\; (x,t)\in M\times \R^*\\\phi(x,X,0)=\frac{1}{2}d^2f_x(X)\quad \text{if}\; (x,X)\in N^M_{\crit(f)}\end{cases}.$$
Furthermore it is straightforward to see that $d\phi=\alpha\in \Gamma(\mathfrak{A}G^0),$ where $d$ as usual denotes the composition of the De Rham derivative with the anchor map.

One has $$h\alpha=hd\phi=d(h\phi)-h\phi \frac{dh}{h}.$$ By \cref{rem V dim 1}, $\frac{dh}{h}$ extends to a smooth section on $K^0$. 

The function $h\phi$ extends to a smooth function on $K^0$. On $M\times \R^*$, one has $$h\phi(x,t)=\frac{f(x)-c}{\sum_ig_i^2(x)+t^2}.$$ This function clearly extends smoothly to $\blup(M\times \R,\crit(f)\times \{0\}).$ This proves the first part concerning the extension of $h\alpha$ to $K^0.$

It is immediate to see that the extension of $h\alpha$ to $K^0$ is equal to $\frac{df}{\sum g_i^2}$ on $M\backslash \crit(f)\times \{0\}$ which is clearly nonzero for every point in $M\backslash \crit(f)\times \{0\}$.

Finally we restrict our attention to the form $h\alpha$ on \begin{align*}
 \big(N^M_{\crit(f)}\backslash \{0\}\times N^M_{\crit(f)}\backslash \{0\}\big)/\R^*\sqcup N^M_{\crit(f)}\times N^M_{\crit(f)}&\rightrightarrows\mathbb{P}(N^{M}_{\crit(f)}\oplus \R)\\&= \mathbb{P}(N^M_{\crit(f)})\sqcup N^M_{\crit(f)}.
\end{align*}

By Morse-Bott lemma, we can suppose that $M=\crit(f)\times L$, with $L$ a vector space and $f=Q$ a quadratic form on $L$. The Lie algebroid of the groupoid  \begin{align*}
 \bigg(\crit(f)\times \crit(f)\bigg)\times \bigg((L\backslash \{0\}\times L\backslash \{0\})/\R^*\sqcup L\times L\bigg)\\ \rightrightarrows \crit(f)\times \mathbb{P}(L\oplus \R)= \crit(f)\times \big(\mathbb{P}(L)\sqcup L\big)
\end{align*} is equal to $T\crit(f)\times \sqcup_{l\in \mathbb{P}(L\oplus \R)}\Hom(l,L)$. It is immediate to see that the value of the extension of $h\alpha$ at $(X,T)\in T\crit(f)\times  \sqcup_{ l\in \mathbb{P}(L)}\Hom(l,TL)$ is given by $$h\alpha(X,T)=\frac{1}{\sum dg_i(a)^2}Q(a,T(a)),$$ where $a$ is any nonzero element in $l$. Hence, it is clear that $h\alpha$ is nowhere zero on $\mathbb{P}(N^M_{\crit(f)}).$ 
\end{proof}
\begin{prop}Let $g$ be a euclidean metric on $\mathfrak{A}K$. The restriction to $\dnc(M,\crit(f))$ of the metric $\frac{g}{h}$ together with the form $\alpha$ satisfies the properties of \cref{keyprop witten}. In particular $d+d^*+c(\alpha)$ is regular with compact resolvent on $C^*(G)$-$C^*$-modules
\end{prop}
\begin{proof}
The metric $g$ is complete by \cref{rem compete metric compact}. It follows from \cref{rem V dim 1} that $\frac{\norm{dh}_g}{h}$ extends to a smooth function on $K^0$, hence it is bounded by compactness of $K^0$. And by \cref{lem h2}, $\norm{h\alpha}_g$ is also bounded. Since $K^0$ is compact, the metric $g$ with the form $h\alpha$ trivially satisfy the conditions of \cref{keyprop witten}. It follows from \cref{comple prop} that $\frac{g}{h}$ and $\alpha$ satisfy the conditions of \cref{keyprop witten}. The result then follows.
\end{proof}


The groupoid $G$ is Morita equivalent to $[-1,1]$. Hence $C^*(G)$ can be replaced by $C([-1,1])$ in the previous proposition.

 \begin{rem}The groupoid $\dnc(M\times M,\crit(f))$ can also be used to obtain a global Witten deformation for Morse-Bott function. For that groupoid the algebroid $\dnc(TM,\crit(f))$ is canonically isomorphic to $\pi^*(TM)$, hence there are no difficulties with the choice of the metric. On the other hand the fiber of its $C^*$-algebra at $0$ is canonically Morita equivalent to $C_0(T^*\crit(f))$. This makes it more difficult to obtain precise statements about eigenvalues.
 \end{rem}

\section{Witten deformation on foliations}\label{section WItten foliations}
Let $F\subseteq TM$ be a subbundle (not necessarily integrable) of the tangent bundle of a closed manifold $M$, $f:M\to \R$ a smooth function. We are interested in the set $$\crit_F(f):=\{x\in M:df_x(F_x)=0\}.$$

\bigskip

Let $x_0\in \crit_F(f)$, $X\in \Gamma^\infty(TM)$, $Y\in \Gamma^\infty(F)$. One defines $$d^2_{x_0}f(X,Y):=(XYf)(x_0).$$ \begin{prop}\label{d2f foliation} The number $d^2_{x_0}f(X,Y)$ only depends on $X(x_0)$ and $Y(x_0)$. In other words $d^2_{x_0}f:T_{x_0}M\times F_{x_0}\to \R$ is a well defined bilinear form. \end{prop}
\begin{proof}
This is clear for $X$. Let $Y'\in \Gamma^\infty(F)$ be another section such that $Y'(x_0)=Y(x_0).$ It follows that $Y'-Y$ could be written as the sum of elements of the form $gZ$, where $g:M\to \R$ is a function that vanishes at $x_0$, and $Z\in \Gamma^\infty(F).$ Hence \begin{align*}
X(Y+gZ)f(x_0)=XYf(x_0)+X(g)(x_0)Zf(x_0)+g(x_0)XZf(x_0)=XYf(x_0),
\end{align*}
where the second term vanishes because $Zf(x_0)=0$ and the third because $g(x_0)=0.$
\end{proof}
\begin{prop}\label{trans more functions foliations}Let $Z=F^{\perp}\subseteq T^*M$. The section $df:M\to T^*M$ is transverse to $Z$ if and only if for every $x\in \crit_F(f)$, the bilinear map $d^2_xf$ is of maximal rank, that is the induced linear map $d^2_{x}f:T_xM\to F_x^*$ is surjective. Furthermore if this is the case, then $\crit_F(f)$ is a smooth manifold whose tangent bundle is $T\crit_F(f)=\ker(d^2_xf).$
\end{prop}
\begin{proof}
Let $h:T^*M\to F^*$ be the map which restricts a linear map on $TM$ to $F$. It is clear that $h\circ df$ is equal to the longitudinal derivative of $f$. Furthermore the inverse image by $h$ of the zero section is equal to $Z$. It follows that $df:M\to T^*M$ is transverse to $Z$ if and only if the longitudinal derivative $d_Ff:M\to F^*$ is transverse to the zero section.

Let $c(t)$ be a path in $M$, with $c(0)$ a point such that $d_Ff_{c(0)}=0$. The derivative of the function $d_Ff_{c(t)}:\R\to F^*$ at zero is then a linear map on $T_{c(0)}F^*=T_{c(0)}M\oplus F^*_{c(0)}.$ The value of this linear map at a vector in $F_{c(0)}$ can be computed using a section of $X$ of $F$. It follows that the derivative of $d_Ff_{c(t)}$ at time $0$ applied to $X(c(0))$ is equal to $\frac{d}{dt}_{|t=0}Xf(c(t))$. This is precisely the definition of $d^2_{c(0)}f(c'(0),X(0))$. The proposition is then clear.
\end{proof}
By Thom's multijet transversality theorem (see \cite[theorem 4.9]{MR0341518}), the transversality condition of \Cref{trans more functions foliations} is satisfied generically. We now fix such a function $f$, and suppose that $F$ is integrable.

\begin{rem} The manifold $\crit_F(f)$ is of complementary dimension to $F$. It is transverse to $F$ at a point $x\in \crit_F(f)$ if and only if the critical point $x$ of the function $f|l_x$ is non degenerate, where $l_x$ is the leaf containing $x.$ In particular, if the foliation doesn't admit a closed transversal, then there exist no smooth function which is Morse on each leaf.\end{rem}

Let $$\mathcal{G}(M,F)=\{(\gamma(0),[\gamma],\gamma(1))|\gamma:[0,1]\to M,\gamma'(t)\in F\,\forall t\in [0,1]\}$$ be the holonomy groupoid of $F$, where $[\gamma]$ denotes the holonomy class of $\gamma$.

Let \begin{align*}
 G=\dnc(\mathcal{G}(M,F),\crit_F(f))\rightrightarrows \dnc(M,\crit_F(f))
\end{align*} be the deformation of foliation groupoid $\mathcal{G}(M,F)$ along the submanifold of its units $\crit_F(f).$ In this section and the next one all deformation to the normal cone are fibered on $[0,1]$. By \Cref{Construction of dnc grouopids}, this is a Lie groupoid whose Lie algebroid is equal to $\dnc(F,\crit_F(f))$. Recall that by \Cref{examps dnc}, $$\mathcal{N}^{\mathcal{G}(M,F)}_{\crit_F(f)}=\{(X,A,Y):X,Y\in \mathcal{N}^M_{\crit_F(f)},A\in F,\; Y-X=A\mod T\crit_F(f) \}.$$

\bigskip

Let $g$ be a Euclidean metric on $F$. The Lie algebroid $\dnc(F,\crit_F(f))$ admits then a Euclidean metric by \Cref{metric dnc E}. On $M\times \{t\}$, it is equal to $\frac{g}{t^2}$, and on $\mathcal{N}^M_{\crit_F(f)}$ it is equal to $g_{|\crit_F(f)}$. This metric is complete because for $t\neq 0$, the metric $g$ is complete on each leaf, and for $t=0$, it is complete by the description of $\mathcal{N}^{\mathcal{G}(M,F)}_{\crit_F(f)}$ given above.

\bigskip

 Let $$\alpha=\dnc(d_Ff):\dnc(M,\crit_F(f))\to \dnc(F^*,\crit_F(f)),$$ where $d_F$ is the longitudinal derivative. The map $\alpha$ is regarded as a section of $$\mathfrak{A}G^*=\dnc(F,\crit_F(f))^*=\dnc(F^*,\crit_F(f)).$$ See \Cref{metric dnc E} for the last equality. On $M\times\{t\}$, $\alpha=\frac{d_Ff}{t^2}$, and on $\mathcal{N}^M_{\crit_F(f)}\times\{0\}$ it is equal to $d^2f$.

\bigskip

Let us show that the hypotheses of \Cref{keyprop witten} hold.
\begin{enumerate}
\item It is clear that the form $\alpha$ is closed.
\item On $M\times\{t\}$, one has $$\alpha^{\#_{\frac{g}{t^2}}}=d_Ff^{\#_g},$$ where $\#$ is the musical isomorphism. Hence $\mathcal{L}_{\alpha^{\#}}$ is independent of $t$. Since the Riemannian metric is multiplied by a scalar, it follows that $\mathcal{L}_{\alpha^{\#}}^*$ doesn't depend on $t$ as well. In other words, the section $\mathcal{L}_{\alpha^{\#}}+\mathcal{L}_{\alpha^{\#}}^*$ on $\dnc(M,\crit_F(f))$ is the pullback of the section $\mathcal{L}_{d_Ff^{\#_g}}+\mathcal{L}_{d_Ff^{\#_g}}^*$ on $M$ using the projection map $\pi:\dnc(M,\crit_F(f))\to M$. Hence the norm of the section $\mathcal{L}_{\alpha^{\#}}+\mathcal{L}_{\alpha^{\#}}^*$ is bounded by its boundness on $M$.
\item On $M\times\{t\}$, one has $$\norm{\alpha}_{\frac{g}{t^2}}=\norm{\frac{d_Ff}{t^2}}_{\frac{g}{t^2}}=\frac{\norm{d_Ff}_g}{t}$$ and on $N^M_{\crit_F(f)}\times\{0\}$, $$\norm{\alpha}_{g}=\norm{d^2f}_{g_a}.$$ Transversality of $d_Ff$ to the zero section is then equivalent to properness of $\norm{\alpha}$ as a function on $\dnc(M,\crit_F(f)).$
\end{enumerate}
\begin{theorem}\label{Wittens thm foliations}The closure of the operator $d+d^*+c(\alpha)$ acting on $C^*(\Lambda_\C\mathfrak{A}G^*)$ is a regular self adjoint operator with a compact resolvent.
\end{theorem}
By \Cref{Wittens thm foliations}, the operator $d+d^*+c(\alpha)$ defines an element in $\KK^0(\C,C^*G).$ By regarding the evaluation  at $0$ and at $1$ of the previous element, one deduces:
\begin{cor}The Euler characteristic $e(F)$ of $\mathcal{F}$ as an element in $\KK(\C,C^*(\mathcal{G}(M,F)))$ can be represented by the element $d+d^*+c(\alpha)$ in $\KK^0(\C,C^*\mathcal{N}^{\mathcal{G}(M,F)}_{\crit_F(f)}).$ More precisely, $$e(F)=\mathrm{Ind}^{\mathcal{G}(M,F)}_{\crit_F(f)}([d+d^*+c(\alpha)]),$$where $$\mathrm{Ind}^{\mathcal{G}(M,F)}_{\crit_F(f)}:K_0\left(C^*\left(N^{\mathcal{G}(M,F)}_{\crit_F(f)}\right)\right)\to K_0(C^*\mathcal{G}(M,F))$$ is defined in \cite{Debord:2017aa}. It is also defined in \cite{MR775126} in a slightly different formalism where it is denoted by $i_!$.
\end{cor}
\begin{rem}It is clear that one can replace $f$ with a $1$-form $\alpha\in \Gamma(F^*)$ such that $d_F\alpha=0$ and $\alpha:M\to F^*$ is transverse to the zero section.
\end{rem}

\subsection{Quantitative result}\label{qualit result sect}
In this section, we will show that when a transverse measure is used to obtain a quantitative result from \cref{Wittens thm foliations}, one obtains Connes-Fack Morse inequalities. More precisely, we suppose that $M$ is a compact smooth manifold, $F$ a foliation of dimension $p$ and codimension $q$ equipped with a holonomy invariant transverse measure $v$. 
We refer to \cite{MR548112,MR679730} for more details on transverse measure from the point of view of Von Neumann algebras. As in \cref{section WItten foliations}, a Euclidean metric $g$ on $F$ is chosen. The measure $\nu$ together with $g$ give rise to a measure on $M$ that will be denoted $\mu=g\times \nu$.

\begin{dfn}\label{nice dfn} A smooth function $f:M\to \R$ is called nice if 
 \begin{enumerate}
\item all longitudinal critical points of $f$ are either Morse or of Birth-death type. See \cite{MR1070701} for the definition of Birth-death singularity.
\item the set of leaves on which $f$ isn't Morse is of $\nu$-measure zero.
\end{enumerate}
\end{dfn}
It is shown in \cite{MR1070701} and \cite{MR1760426} that any smooth function can be $C^0$ approximated by nice functions. Furthermore if $\codim(F)\leq \dim(F)$ then the approximation can be supposed $C^1$. In what follows we suppose that $f$ is a nice function.

 Morse lemma for a birth-death critical point $x_0$ states the following; there exists a foliated chart where $x_0=0\in U\subseteq \R^p\times \R^q$ such that $$f(x,y)=f(0,y)+\frac{1}{3}x_1^3-y_1x_1+ x_2^2\dots+x^2_{p-i}-x^2_{p-i+1}-\dots -x_p^2,$$ and $\crit_F(f)=\{(x,y):y_1=x_1^2,x_2,\dots,x_p=0\}$. Here $i$ is called the index of $x_0$.  We will suppose that $g$ is locally Euclidean in such foliated charts. This is a nonrestrictive condition since everything below is about densities and not Euclidean metrics. The eigenvalues of the longitudinal Hessian of $f$ at a point $(x,y)\in \crit_F(f)$ are equal to $2$ with multiplicity $p-1-i$, $-2$ with multiplicity $i$, and $2x_1$ with multiplicity $1$.

\smallskip

Such local coordinates will be called nice local charts. For convenience we will suppose that $U$ contains $[-1,1]^{n}.$

\smallskip

 We denote by $G=\dnc(\mathcal{G}(M,F),\crit_F(f)).$ By \cite{MR548112}, for $t\neq 0$, $\nu$ defines a semi-finite lower semi-continuous trace on $C^*(G)$ by the formula, $$\tau_t(\phi)=\frac{1}{t^{\dim(p)}}\int_{M}\phi(\cdot,t) d\mu,$$ where $\phi\in C_c^\infty(G)$, and the integral is on the subspace of the space of objects $M\times \{t\}\subseteq \dnc(M,\crit_F(f))$. Notice that the factor $\frac{1}{t^{\dim(p)}}$ comes from the fact that we work in this article with functions instead of densities. 
 
\smallskip

 For $t=0$, let $ \critr_F(f)\subseteq \crit_F(f)$ be the open dense subset of points where $\crit_F(f)$ is transversal to the foliation. On $\critr_F(f)$, the vector bundle $F_{|\critr_F(f)}$ is transverse to $T\critr_F(f)$. We identify the vector bundle $N^{M}_{\critr_F(f)}$ with $F_{|{\critr_F(f)}}$. The trace $\tau_{0}$ on $C_c^\infty(\mathcal{N}^{\mathcal{G}(M,F)}_{\crit_F(f)})$ is defined by first restricting a function $\phi\in C_c^\infty(\mathcal{N}^{\mathcal{G}(M,F)}_{\crit_F(f)})$ to the open subset $\mathcal{N}^M_{\critr_F(f)}$ of the space of objects $\mathcal{N}^M_{\crit_F(f)}$, then using the formula $$\tau_{0}(\phi)=\int_{x\in \critr_F(f)}\int_{F_x}\phi dg_{x}d\nu.$$ Equivalently the restriction of $C^*(\mathcal{N}^{\mathcal{G}(M,F)}_{\crit_F(f)})$ to $\critr_F(f)$ is equal to a fiber bundle over $\critr_F(f)$ with the fiber over each point $x\in \critr_F(f)$ equal to $\mathcal{K}(L^2(F_x))$. The trace is then the canonical trace on each fiber followed with the trace against the transverse measure on the base.
 
 \smallskip
 
 This trace extends to the $C^*$-algebra. On the other hand this trace is not necessarily finite $C^\infty_c(G)$. In fact in \cref{continuity trace lemma}, we prove that $\tau_0$ is finite if and only if for every birth-death singularity $x_0$, and nice local chart around $x_0$, the following integral is finite $$\int_{\{y\in [-1,1]^q:y_1>0\}}\frac{1}{\sqrt{y_1}}<+\infty.$$
  
Let us remark that in a nice local chart the singular values of the longitudinal Hessian of $f$ at a point $(x,y)\in \crit_F(f)$ are equal to $2$ with multiplicity $p-1$ and $2\sqrt{y_1}$ with multiplicity $1$. It follows that the integrability of the function $\frac{1}{\sqrt{y_1}}$ around every birth-death singularity is equivalent to the finiteness of the integral 
   $$\int_{\critr_F(f)}\frac{1}{E_{p}(d^2_Ff)}d\nu<\infty,$$ where $E_p(d^2_Ff)$ is the $p$th singular value of the longitudinal Hessian (the restriction of $d^2f$ defined in \cref{d2f foliation} to $F\times F$). 
\begin{prop}\label{continuity trace lemma}
\begin{enumerate}
\item The trace $\tau_0$ is finite if and only if  $\int_{\critr_F(f)}\frac{1}{E_{p}(d^2_Ff)}d\nu<\infty$.
\item Suppose that $\int_{\critr_F(f)}\frac{1}{E_{p}(d^2_Ff)}d\nu<\infty$, and let $\phi$ be a smooth function with compact support on $ \dnc(\mathcal{G}(M,F),\crit_F(f))$. The function $t\to \tau_{t}(\phi)$ is continuous .
\end{enumerate}
\end{prop}
\begin{proof}
By a partition of unity argument, it suffices to verify the proposition around a local neighbourhood of every critical point $x\in \crit_F(f)$. The case of Morse singularity is straightforward. We will only do the case of a birth-death singularity. Let $x_0\in \crit_F(f)$ be a birth-death singularity. In a nice chart around $x_0$, we have $$f(x,y)=f(0,y)+\frac{1}{3}x_1^3-y_1x_1+ x_2^2\dots+x^2_{p-i}-x^2_{p-i+1}-\dots -x_p^2$$ and $\crit_F(f)=\{(x,y):y_1=x_1^2,x_2,\dots,x_p=0\}$.  We can assume that $\phi\in C^\infty_c(\dnc(\R^p\times \R^q, \crit_F(f)))$ is a smooth function. Being smooth with compact support is equivalent to the existence of a smooth function with compact support $\tilde{\phi}\in C^\infty_c(\R^p\times \R^q\times \R)$ such that for $t\neq 0$, we have $$\phi(x,y,t)=\tilde{\phi}(x_1,\frac{x_2}{t},\dots,\frac{x_p}{t},\frac{y_1-x_1^2}{t},y_2,\dots,y_q,t).$$  

Notice that $\nu(\{0\}\times \R^{q-1})=0$ by the conditions on $f$.

In the following equation we will use the notation $x',y'$ for the vectors $(x_2,\dots,x_p)$, $(y_1,\dots,y_q)$. For $t\neq 0$, one has
\begin{align*}
\tau_t(\phi)&=\int_{\{y\in \R^q:y_1>0\}}\int_{\R^p}\frac{1}{t^p}\phi(x,y,t)dxd\nu(y)\\&=\int_{\{y\in \R^q:y_1>0\}}\int_{\R^p}\frac{1}{t^p}\tilde{\phi}(x_1,\frac{x'}{t},\frac{y_1-x_1^2}{t},y',t)dxd\nu(y)\\&=\int_{\{y\in \R^q:y_1>0\}}\int_{\R^p}\frac{1}{t}\tilde{\phi}(x_1,x',\frac{y_1-x_1^2}{t},y',t)dxd\nu(y)
\end{align*}
In the last identity we divide the integral into the sum of the integral over the region where $x_1\geq 0$ and the integral over the complement and then used the change of variables $z=\frac{y_1-x_1^2}{t}$, to obtain \begin{align*}
\tau_t(\phi)&=\int_{\{y\in \R^q:y_1>0\}}\int_{\R^{p-1}}\int_{-\infty}^{\frac{y_1}{t}}\frac{1}{2\sqrt{y_1-tz}}\tilde{\phi}(\sqrt{y_1-tz},x',z,y',t)dzdx'd\nu(y)\\&+\int_{\{y\in \R^q:y_1>0\}}\int_{\R^{p-1}}\int_{-\infty}^{\frac{y_1}{t}}\frac{1}{2\sqrt{y_1-tz}}\tilde{\phi}(-\sqrt{y_1-tz},x',z,y',t)dzdx'd\nu(y)
\end{align*}

Similarly one checks as well that \begin{align*}
\tau_0(\phi)&=\int_{\{y\in \R^q:y_1>0\}}\int_{\R^p}\frac{1}{2\sqrt{y_1}}\tilde{\phi}(\sqrt{y_1},x',z,y',0)dx'dzd\nu(y)\\&+\int_{\{y\in \R^q:y_1>0\}}\int_{\R^p}\frac{1}{2\sqrt{y_1}}\tilde{\phi}(-\sqrt{y_1},x',z,y',0)dx'dzd\nu(y).
\end{align*}
The first part of the proposition is then clear and the second follows from the dominated convergence theorem.
\end{proof}
\begin{prop}\label{approx fun}For any nice function $f:M\to \R$ and $\epsilon>0$, there exists a nice function $h:M\to \R$ such that \begin{enumerate}
\item $\int_{\critr_F(h)}\frac{1}{E_p(d^2_Fh)}d\nu<+\infty$.
\item $|\nu(\crit^i_F(h))-\nu(\crit^i_F(f))|\leq \epsilon$, where $\crit^i_F(f)$, $\crit^i_F(h)$ is the set of Morse critical points of index $i$ of $f$ and $h$ respectively.
\end{enumerate}
\end{prop}
\begin{proof}
Let $f$ be a nice function. We will modify the function $f$ in local charts around birth-death singularities to make the condition  $\int_{\critr_F(f)}\frac{1}{E_p(d^2_Ff)}d\nu<\infty$ hold.

In a nice local chart of a critical point of birth-death singularity type, one has \begin{align*}
\int_{-1}^{1}\int_{-x}^{1}\frac{1}{\sqrt{y+x}}d\nu(y)dx&=\int_{-1}^{1}\int_{-y}^{1}\frac{1}{\sqrt{y+x}}dxd\nu(y)\\&=\int_{-1}^{1}2\sqrt{y+1}d\nu(y)<+\infty.
\end{align*}
Equally well $\sum_{x\in [-1,1]}\nu(\{x\}\times [-1,1]^{q-1})\leq \nu([-1,1]^q)<+\infty$. It follows that for almost all $\epsilon\in [-1,1]$ with respect to the Lebesgue measure, the function $\frac{1}{\sqrt{y-\epsilon}}$ is $\nu$-integrable and $\nu(\{\epsilon\}\times [-1,1]^{q-1})=0$. 
If one takes such $\epsilon$, and $\phi:\R^n\to \R^n$ any diffeomorphism which is equal to the identity outside $[-1,1]^{n}$ and which is equal on $[-\frac{1}{2},\frac{1}{2}]^{n}$ to the diffeomorphism $$(y,x)\to (y_1-\epsilon,y_2,\dots,x),$$ then the function $f\circ\phi$ is a nice function such that $\int_{\critr_F({f\circ \phi})}\frac{1}{E_p(d^2_F(f\circ \phi))}d\nu$ is finite around the critical point. Doing the above procedure for a finite cover of birth-death singularities finishes the proof.
\end{proof}
\begin{cor}[Connes-Fack Morse inequalities]Let $f$ be a nice function, $c_i=\nu(\crit_\mathcal{F}^i(f))$ the $\nu$-measure of critical points of index $i$, and $\beta_i$ the $\nu$-measure of $\ker(\Delta^i)$. One has $$\sum_{i=0}^k(-1)^{k-i}c_i\geq \sum_{i=0}^k(-1)^{k-i}\beta_i$$
\end{cor}
\begin{proof}
By  \cref{approx fun}, it is enough to prove the corollary under the extra hypothesis that $\int_{\critr_F(f)}\frac{1}{E_p(d^2f_x)}d\mu(x)$ is finite. Let $\phi\in C^0([0,+\infty[)$ be a decreasing continuous function with compact support which is equal to $1$ on the interval $[0,1]$. Algebraic Morse inequalities imply that $$\sum_{i=0}^k(-1)^{k-i}\tau_t(\phi(\Delta^i_{t,f}))\geq \sum_{i=0}^k(-1)^{k-i}\tau_t(\ker(\Delta^i_{t,f}))$$ for any $t\neq 0$, where $\Delta^{i}_{t,f}$ is the Witten deformation constructed in \cref{Wittens thm foliations} at time $t$ acting on forms of degree $i$. Since the multiplication $e^{\frac{f}{t}}$ gives an isomorphism from the deformed De Rham complex at time $t$ with the De Rham complex. It follows that $\tau_t(\ker(\Delta^i_{t,f}))=\beta_i$.

 It follows from \cref{continuity trace lemma}, that \cref{cont trace pro} can be applied to the traces $\tau_t$. Hence \cref{Wittens thm foliations} together with \cref{cont trace pro} imply that $t\to \tau_t(\phi(\Delta^i_{t,f}))$ is a continuous function. It follows that $$\sum_{i=0}^k(-1)^{k-i}\tau_0(\phi(\Delta^i_{0,f}))\geq \sum_{i=0}^k(-1)^{k-i}\beta_i.$$ Applying the previous inequality to $\phi(\frac{\cdot}{\epsilon})$, and taking the limit when $\epsilon$ goes to zero one obtains $$\sum_{i=0}^k(-1)^{k-i}\tau_t(\ker(\Delta^i_{0,f})))\geq \sum_{i=0}^k(-1)^{k-i}\beta_i.$$ By definition one has $\tau_t(\ker(\Delta^i_{0,f}))=c_i.$
 \end{proof}
 \begin{rem}\begin{enumerate}
 \item Notice that the above argument works equally well with a longitudinally Novikov $1$-form, a $1$-closed which is longitudinally closed and which locally is the longitudinal derivative of a nice function. The only difference between the Novikov $1$-form case and the above is that one would obtain the following inequalities \begin{theorem}Let $\alpha$ be a longitudinally nice Novikov $1$-form, $c_i=\nu(\crit_\mathcal{F}^i(\alpha))$ the $\nu$-measure of critical points of index $i$, and $\beta_i(t)$ the $\nu$-measure of $\ker(\Delta_t^i)$ where $\Delta_t=(d+d^*+tc(\alpha))^2$. 

One has $$\sum_{i=0}^k(-1)^{k-i}c_i\geq \limsup_{t\to 0}\sum_{i=0}^k(-1)^{k-i}\beta_i(t)$$

 \end{theorem}
In the classical case of a compact manifold the spectrum is discrete and $\beta_i(t)$ is integar valued which implies that $\beta_i(t)$ is constant for $t$ small enough. The limit is then called the Novikov betti number of $\alpha$. In the foliation case, I don't see any reason for $\beta_i(t)$ to converge as $t\to 0$.
 \end{enumerate}
 \end{rem}

\section{Continuous family of traces}
Recall that a lower-semi continuous trace on a $C^*$-algebra $A$ is a lower semi-continuous function $\tau:A^+\to [0,+\infty]$ such that \begin{enumerate}
\item for every $a,b\in A^+$, $\lambda,\mu\in \R^+$, $\tau(\lambda a+\mu b)=\lambda \tau(a)+\mu \tau(b)$.
\item for every unitary $u\in A$, $\tau(uau^*)=\tau(a).$
\end{enumerate}

It is a classical theorem (see for example \cite{MR0458185}) that for a trace $\tau$, the span of elements $\{x\in A^+: \tau(a)<+\infty\}$ is an ideal on which $\tau$ admits a unique linear extension. In  \cite{MR0458185}, the following inequality is also proved $$|\tau(xy)|\leq \tau(|x|)\norm{y}\forall x,y\in A.$$

\begin{prop}\label{cont trace pro}For each $t\in [0,1]$, let $\tau_t:A^+\to [0,+\infty]$ be a lower semi-continuous trace, $\mathcal{A}\subseteq A$  a dense $*$-subalgebra such that for every $x\in \mathcal{A}$, $t\to \tau_t(x)$ is a finite continuous function. Then
\begin{enumerate}

\item if $a\in A$ such that $\sup_{t\in [0,1]}\tau_t(|a|)<+\infty$, then for every $b\in A$, $t\to \tau_t(ab)$ is a continuous function.
\item If $a\in A^+$, $f\in C^0([0,+\infty[)$ a continuous function which vanishes on a neighbourhood of $0$, then $t\to \tau_t(f(a))$ is a finite continuous function.

It follows from $1$, that if $c\in A$, then the function $t\to \tau_t(f(a)c)$ is continuous.

\item If $a\in A^+$, then $t\to \tau_t(a)$ is lower semi-continuous.
\end{enumerate} 
\end{prop}
\begin{proof}
\begin{enumerate}
 
\item By writing $b=b_1+b_2+ib_3-ib_4$ with $b_i$ positive, we can suppose that $b$ is positive. Let $a_n$, $b_n$ be sequence of elements of $\mathcal{A}$ that converge to $a$ and $b$ respectively. Furthermore we can suppose that $b_n$ is positive for all $n$. Since $\mathcal{A}$ is an algebra, $t\to \tau_t(a_nb_m)$ is a finite continuous for every $n,m$. It follows from the inequality $$|\tau_t(a_nb_m)-\tau_t(ab_{m})|\leq \tau_t(b_m)\norm{a_n-a}$$ and the fact that $t\to \tau_t(b_m)$ is a finite continuous function that as $n\to \infty$, $t\to \tau_t(a_nb_m)$ converges uniformly to $t\to \tau_t(ab_m)$. Hence $t\to \tau_t(ab_m)$ is a finite continuous function for every $m$. It follows from the inequality $$|\tau_t(ab_m)-\tau_t(ab)|\leq \tau_t(|a|)\norm{b_m-b}\leq \sup_{t\in [0,1]}\tau_t(|a|)\norm{b_m-b}$$ that $t\to \tau_t(ab_m)$ converges uniformly to $t\to \tau_t(a b)$. Hence $t\to \tau_t(ab)$ is a finite continuous function.
\item We will use notation from \cite{MR840845}. In particular, for a trace $\tau$, $a\in A^+$ we will denote by $$E_s^{\tau}(a):=\inf\{\alpha\in ]0,+\infty[,\tau(P_{]\alpha,+\infty[}(a)\leq s\}$$ the $s$-number of $a$ with respect to the trace $\tau$, where $P$ is the spectral projection.

It is proved in \cite{MR840845}, that $$|E_s^{\tau}(a)-E_s^{\tau}(b)|\leq \norm{a-b},\; \tau(f(a))=\int_{0}^\infty f(E_s^{\tau}(a))ds$$ for any continuous function $f\in C^0([0,+\infty[)$, $a,b\in A^+$. 



Let $\alpha>0$, $\phi_\alpha(x)=\max(x-\alpha,0)\in C^0([0,+\infty[).$ If $b\in \mathcal{A}^+$ such that $\norm{a-b}\leq \alpha$. Since $\phi_{\alpha}$ is increasing and $E_{s}^{\tau_t}(a)\leq E_s^{\tau_t}(b)+\alpha$, one has  \begin{align*}
\tau_t(\phi_\alpha(a))&=\int_{0}^{+\infty}\phi_{\alpha}(E_s^{\tau_t}(a))ds\\&\leq \int_{0}^{+\infty}\phi_{\alpha}(E_s^{\tau_t}(b)+\alpha)ds\\&=\int_{0}^{+\infty}E_s^{\tau_t}(b)ds=\tau_t(b)
\end{align*}
It follows that $\sup_{t\in [0,1]}\tau_t(\phi_\alpha(a))<+\infty$.

Let $f\in C^0([0,+\infty[)$ be a continuous function which vanishes on $[0,2\alpha[.$ One can write $f=\phi_{\alpha}g$ with $g$ a continuous function such that $g(0)=0$. It follows from $1$ and the above that $t\to \tau_t(f(a))$ is continuous.


\item This follows from $2$ and the monotone convergence theorem (see  \cite{MR840845}) by writing the function $f(x)=x$ as the supremum of functions satisfying the conditions of $3$.
\end{enumerate}
\end{proof}
\begin{rem}Part $2$ of \cref{cont trace pro} equivalently says that for every $x$ in the Pedersen ideal \footnote{Smallest dense two-sided ideal of $A$.}, $t\to \tau_t(x)$ is a finite continuous function.
\end{rem}
 \bibliography{../Biblatex}

\begin{thebibliography}{10}

\bibitem{MR715325}
S. Baaj and P. Julg.
\newblock Th{\'e}orie bivariante de {K}asparov et op{\'e}rateurs non born{\'e}s
  dans les {$C^{\ast} $}-modules hilbertiens.
\newblock {\em C. R. Acad. Sci. Paris S{\'e}r. I Math.}, 296(21):875--878,
  1983.

\bibitem{MR852155}
J.-M. Bismut.
\newblock The {W}itten complex and the degenerate {M}orse inequalities.
\newblock {\em J. Differential Geom.}, 23(3):207--240, 1986.

\bibitem{MR1656031}
B. Blackadar.
\newblock {\em {$K$}-theory for operator algebras}, volume~5 of {\em
  Mathematical Sciences Research Institute Publications}.
\newblock Cambridge University Press, Cambridge, second edition, 1998.

\bibitem{MR0369890}
P.~R. Chernoff.
\newblock Essential self-adjointness of powers of generators of hyperbolic
  equations.
\newblock {\em J. Functional Analysis}, 12:401--414, 1973.

\bibitem{MR679730}
A.~Connes.
\newblock A survey of foliations and operator algebras.
\newblock In {\em Operator algebras and applications, {P}art {I} ({K}ingston,
  {O}nt., 1980)}, volume~38 of {\em Proc. Sympos. Pure Math.}, pages 521--628.
  Amer. Math. Soc., Providence, R.I., 1982.

\bibitem{MR775126}
A.~Connes and G.~Skandalis.
\newblock The longitudinal index theorem for foliations.
\newblock {\em Publ. Res. Inst. Math. Sci.}, 20(6):1139--1183, 1984.

\bibitem{MR548112}
A. Connes.
\newblock Sur la th{\'e}orie non commutative de l'int{\'e}gration.
\newblock In {\em Alg{\`e}bres d'op{\'e}rateurs ({S}{\'e}m., {L}es
  {P}lans-sur-{B}ex, 1978)}, volume 725 of {\em Lecture Notes in Math.}, pages
  19--143. Springer, Berlin, 1979.

\bibitem{MR823176}
A. Connes.
\newblock Noncommutative differential geometry.
\newblock {\em Inst. Hautes {\'E}tudes Sci. Publ. Math.}, (62):257--360, 1985.

\bibitem{MR1303779}
A. Connes.
\newblock {\em Noncommutative geometry}.
\newblock Academic Press, Inc., San Diego, CA, 1994.

\bibitem{MR2276915}
A. Connes and T. Fack.
\newblock Morse inequalities for foliations.
\newblock In {\em {$C^\ast$}-algebras and elliptic theory}, Trends Math., pages
  61--72. Birkh{\"a}user, Basel, 2006.

\bibitem{Debord:2017aa}
C. Debord and G. Skandalis.
\newblock Blowup constructions for {L}ie groupoids and a {B}outet de {M}onvel
  type calculus.
\newblock {\em Arxiv:1705.09588}.

\bibitem{ClaireGeorges}
C. Debord and G. Skandalis.
\newblock Lie groupoids, pseudodifferential calculus and index theory.
\newblock {\em Arxiv:1907.05258}.

\bibitem{MR0458185}
J. Dixmier.
\newblock {\em {$C\sp*$}-algebras}.
\newblock North-Holland Publishing Co., Amsterdam-New York-Oxford, 1977.
\newblock Translated from the French by Francis Jellett, North-Holland
  Mathematical Library, Vol. 15.

\bibitem{MR1760426}
Y.~M. Eliashberg and N.~M. Mishachev.
\newblock Wrinkling of smooth mappings. {II}. {W}rinkling of embeddings and
  {K}. {I}gusa's theorem.
\newblock {\em Topology}, 39(4):711--732, 2000.

\bibitem{MR840845}
T. Fack and H. Kosaki.
\newblock Generalized {$s$}-numbers of {$\tau$}-measurable operators.
\newblock {\em Pacific J. Math.}, 123(2):269--300, 1986.

\bibitem{MR0246142}
I.~C. Gohberg and M.~G. Kreun.
\newblock {\em Introduction to the theory of linear nonselfadjoint operators}.
\newblock Translated from the Russian by A. Feinstein. Translations of
  Mathematical Monographs, Vol. 18. American Mathematical Society, Providence,
  R.I., 1969.

\bibitem{MR0341518}
M.~Golubitsky and V.~Guillemin.
\newblock {\em Stable mappings and their singularities}.
\newblock Springer-Verlag, New York-Heidelberg, 1973.
\newblock Graduate Texts in Mathematics, Vol. 14.

\bibitem{MR1050489}
M. Hilsum.
\newblock Fonctorialit{\'e} en {$K$}-th{\'e}orie bivariante pour les
  vari{\'e}t{\'e}s lipschitziennes.
\newblock {\em $K$-Theory}, 3(5):401--440, 1989.

\bibitem{MR1070701}
K. Igusa.
\newblock {$C^1$} local parametrized {M}orse theory.
\newblock {\em Topology Appl.}, 36(3):209--231, 1990.

\bibitem{MR918241}
G.~G. Kasparov.
\newblock Equivariant {$KK$}-theory and the {N}ovikov conjecture.
\newblock {\em Invent. Math.}, 91(1):147--201, 1988.

\bibitem{MR1435704}
D. Kucerovsky.
\newblock The {$KK$}-product of unbounded modules.
\newblock {\em $K$-Theory}, 11(1):17--34, 1997.

\bibitem{MR1325694}
E.~C. Lance.
\newblock {\em Hilbert {$C^*$}-modules}, volume 210 of {\em London Mathematical
  Society Lecture Note Series}.
\newblock Cambridge University Press, Cambridge, 1995.
\newblock A toolkit for operator algebraists.

\bibitem{Lesch:2018aa}
M. Lesch and B. Mesland.
\newblock Sums of regular selfadjoint operators in hilbert-C*-modules.
\newblock 03 2018.

\bibitem{MR3669118}
J.-M. Lescure, D. Manchon, and S. Vassout.
\newblock About the convolution of distributions on groupoids.
\newblock {\em J. Noncommut. Geom.}, 11(2):757--789, 2017.

\bibitem{MR896907}
K.~Mackenzie.
\newblock {\em Lie groupoids and {L}ie algebroids in differential geometry},
  volume 124 of {\em London Mathematical Society Lecture Note Series}.
\newblock Cambridge University Press, Cambridge, 1987.

\bibitem{MR1467076}
B. Monthubert and F. Pierrot.
\newblock Indice analytique et groupo\"\i des de {L}ie.
\newblock {\em C. R. Acad. Sci. Paris S{\'e}r. I Math.}, 325(2):193--198, 1997.

\bibitem{MR1451874}
M. Morse.
\newblock {\em The calculus of variations in the large}, volume~18 of {\em
  American Mathematical Society Colloquium Publications}.
\newblock American Mathematical Society, Providence, RI, 1996.
\newblock Reprint of the 1932 original.

\bibitem{MR1687747}
V. Nistor, A. Weinstein, and P. Xu.
\newblock Pseudodifferential operators on differential groupoids.
\newblock {\em Pacific J. Math.}, 189(1):117--152, 1999.

\bibitem{MR0133785}
K. Nomizu and H. Ozeki.
\newblock The existence of complete {R}iemannian metrics.
\newblock {\em Proc. Amer. Math. Soc.}, 12:889--891, 1961.

\bibitem{MR548006}
G. Pedersen.
\newblock {\em {$C^{\ast} $}-algebras and their automorphism groups}, volume~14
  of {\em London Mathematical Society Monographs}.
\newblock Academic Press, Inc. [Harcourt Brace Jovanovich, Publishers],
  London-New York, 1979.

\bibitem{MR941624}
J. Pradines.
\newblock Remarque sur le groupo\"\i de cotangent de {W}einstein-{D}azord.
\newblock {\em C. R. Acad. Sci. Paris S{\'e}r. I Math.}, 306(13):557--560,
  1988.

\bibitem{MR0493419}
M. Reed and B. Simon.
\newblock {\em Methods of modern mathematical physics. {I}. {F}unctional
  analysis}.
\newblock Academic Press, New York-London, 1972.

\bibitem{MR584266}
J. Renault.
\newblock {\em A groupoid approach to {$C^{\ast} $}-algebras}, volume 793 of
  {\em Lecture Notes in Mathematics}.
\newblock Springer, Berlin, 1980.

\bibitem{MR2227132}
S. Vassout.
\newblock Unbounded pseudodifferential calculus on {L}ie groupoids.
\newblock {\em J. Funct. Anal.}, 236(1):161--200, 2006.

\bibitem{MR683171}
E. Witten.
\newblock Supersymmetry and {M}orse theory.
\newblock {\em J. Differential Geom.}, 17(4):661--692 (1983), 1982.

\end{thebibliography}
\bibliographystyle{plain}

\end{document}